\documentclass[12pt,reqno]{amsart}

\usepackage{amssymb,amsmath}
\usepackage{paralist}
\usepackage{dsfont}

\usepackage{lmodern} 
\usepackage{microtype}

\usepackage[a4paper, hmargin={2cm}]{geometry}

\usepackage{amsthm}

\usepackage[english]{babel} 
\usepackage{csquotes} 
\usepackage{xcolor}
\usepackage{todonotes}
\usepackage{mathtools}

\usepackage[colorlinks=true, citecolor=blue, urlcolor=violet, breaklinks=true]{hyperref}

\newtheorem{theorem}{Theorem}[section]

\newtheorem{proposition}[theorem]{Proposition}
\newtheorem{lemma}[theorem]{Lemma}
\newtheorem{corollary}[theorem]{Corollary}

\theoremstyle{definition}
\newtheorem{definition}[theorem]{Definition}
\newtheorem{example}[theorem]{Example}
\newtheorem{remark}[theorem]{Remark}

\DeclareMathOperator{\powercone}{Pow}

\DeclareMathOperator{\indco}{indco}
\DeclareMathOperator{\conv}{conv}
\DeclareMathOperator{\cone}{co}

\DeclareMathOperator{\interior}{int}
\DeclareMathOperator{\relint}{ri}

\DeclareMathOperator{\supp}{supp}

\DeclareMathOperator{\cl}{cl}
\DeclareMathOperator{\range}{range}

\DeclareMathOperator{\rec}{rec}

\newcommand{\relentr}{D}
\newcommand{\e}{\mathrm{e}}

\newcommand{\N}{\mathbb{N}}
\newcommand{\Nsf}{\mathsf{N}}
\newcommand{\R}{\mathbb{R}}

\DeclareMathAlphabet{\mymathbb}{U}{BOONDOX-ds}{m}{n}
\newcommand{\zerob}{\mymathbb{0}}
\newcommand{\oneb}{\mathds{1}}
\newcommand{\cA}{\mathcal{A}}

\author{Riley Murray}
\address{Riley Murray:  University of California, Berkeley - Department of Electrical Engineering and Computer Sciences ({\sf{rjmurray@berkeley.edu}}).
\textit{Previously}, California Institute of Technology - Department of Computing and Mathematical Sciences}

\author{Helen Naumann}
\author{Thorsten Theobald}

\address{Helen Naumann, Thorsten Theobald:
	Goethe-Universit\"at, FB 12 -- Institut f\"ur Mathematik,
	Postfach 11 19 32, D--60054 Frankfurt am Main, Germany 
	({\sf{\{naumann,theobald\}@math.uni-frankfurt.de}})}


\title[]{Sublinear circuits and the constrained signomial nonnegativity problem}

\date{\today}

\begin{document}

\begin{abstract}
Conditional Sums-of-AM/GM-Exponentials (conditional SAGE) is a decomposition method to prove nonnegativity of a signomial or polynomial over some subset $X$ of real space.
    In this article, we undertake the first structural analysis of conditional SAGE signomials for convex sets $X$.
    We introduce the $X$-circuits of a finite subset $\cA \subset \R^n$, which generalize the simplicial circuits of the 
    affine-linear matroid induced by $\cA$ to a constrained setting. 
    The $X$-circuits serve as the main tool in our analysis and 
    exhibit particularly rich combinatorial properties for polyhedral $X$, in which case the set of $X$-circuits is comprised of one-dimensional cones of suitable polyhedral fans.
 
    The framework of $X$-circuits transparently reveals when an $X$-nonnegative conditional AM\slash GM-exponential can in fact be further decomposed as a sum of simpler $X$-nonnegative signomials.
    We develop a duality theory for $X$-circuits with connections to geometry of sets that are convex according to the geometric mean. This theory provides an optimal power cone reconstruction of conditional SAGE signomials when $X$ is polyhedral.
    In conjunction with a notion of reduced $X$-circuits, the duality theory facilitates a characterization of the
    extreme rays of conditional SAGE cones.
    
    Since signomials under logarithmic variable substitutions give polynomials, our results also have implications for
    nonnegative polynomials and polynomial optimization.
\end{abstract}

\maketitle

\section{Introduction}

Given a finite subset $\mathcal{A} \subset \R^n$, 
a signomial on $\mathcal{A}$ is a real-linear combination
\begin{equation}
\label{eq:signomial1}
f = \sum_{\alpha \in \cA} c_\alpha \e^{\alpha} \text{ of basis functions } \e^{\alpha}(x) \coloneqq \exp(\alpha^T x)
\end{equation}
with coefficients $c = (c_{\alpha})_{\alpha\in\cA}$ .
Signomials are a fundamental class of functions with applications in,
but not limited to, chemical reaction networks \cite{Mller2015,Mller2019}, aircraft design optimization \cite{Ozturk2019-sig-aero,York2018a-sig-aero}, and epidemiological process control \cite{Nowarzi2017-epidemic,Precadio2014-epidemic}.
We refer the reader to \cite{epr-2020} and its references for the manifold occurrences of signomials in pure and applied mathematics and to \cite[\S 1.1]{MurrayPhD} for an abridged history of signomial modeling which begins with geometric programming.
Signomials are often considered under a logarithmic change of variables $y \mapsto f(\log y) = \sum_{\alpha\in\cA}c_{\alpha}\prod_{i=1}^n y_i^{\alpha_i}$, so that for $\cA \subset \N^n$ one obtains polynomials over the positive orthant $\R^{n}_{++}$.

A basic question one might ask of a signomial is when the coefficients $c = (c_{\alpha})_{\alpha\in\cA}$ are such that $f$ is globally nonnegative.
Framing this question in terms of a signomial's coefficients affords direct connections to polynomials.
If the exponent vectors $\cA$ are contained in $\N^n$, then $f$ is nonnegative on $\R^n$ if and only if the polynomial $y \mapsto \sum_{\alpha\in\cA}c_{\alpha}\prod_{i=1}^n y_i^{\alpha_i}$ is nonnegative on the nonnegative orthant $\R^n_+$.
Deciding such nonnegativity problems is NP-hard in general \cite{Murty1987}.
However, several researchers have developed sufficient conditions for nonnegativity based on the AM/GM inequality.
In contrast to the well-known Sums-of-Squares nonnegativity certificates in the polynomial setting (see, e.g., \cite{lasserre-positive-polynomials,powers-woermann-1998}), the techniques based on the AM/GM inequality are not tied to the notion of a polynomial's degree, and hence apply to general signomials.
The earliest results here are due to Reznick \cite{reznick-1989}, with a resurgence marked by the works of Pantea, Koeppl, and Craciun \cite{Pantea2012}, Iliman and de Wolff \cite{iliman-dewolff-resmathsci}, and Chandrasekaran and Shah \cite{chandrasekaran-shah-2016}.
Whether considered for signomials or polynomials, such techniques have appealing forms of sparsity preservation in the proofs of nonnegativity \cite{mcw-2018,wang-nonneg}.

In this article, we are concerned with the question of when a signomial 
on exponents $\cA$ is \emph{$X$-nonnegative} (i.e., nonnegative on $X$) for a convex set $X$.
In important progress on this
question, Murray, Chandrasekaran and Wierman have proposed an
extension of the \textit{Sums-of-AM/GM-Exponentials} or \textit{SAGE} approach to global nonnegativity \cite{chandrasekaran-shah-2016}, which goes by the name \textit{conditional SAGE} \cite{mcw-2019}.
The method works as follows: if a signomial $f$ of the form~\eqref{eq:signomial1}
has at most one negative coefficient $c_{\beta}$, i.e.,
\[
  f \ = \ \sum_{\alpha \in \mathcal{A} \setminus \{\beta\}} c_{\alpha}  \e^{\alpha}
    + c_{\beta} \e^{\beta} \quad \text{ with } \quad c_{\alpha} \ge 0 \; 
    \text{ for all } \alpha \in \mathcal{A} \setminus \{\beta\}
\]
then we may divide out the corresponding basis function $\e^{\beta}$ to obtain a new signomial $g = \sum_{\alpha\in\cA}c_{\alpha}\e^{[\alpha-\beta]}$ without affecting nonnegativity.
Because $g$ is the sum of a signomial with all nonnegative coefficients (a \textit{posynomial}) and a constant, it is convex by construction, its $X$-nonnegativity can be decided by applying the principle of strong duality in convex optimization.
The outcome of this duality argument is that $f$ is $X$-nonnegative if and only if there exists a dual variable $\nu = (\nu_{\alpha})_{\alpha\in\cA}$ that satisfies a certain relative entropy inequality in $\nu$, $c$, and the support function of $X$ 
(see Proposition~\ref{pr:rel-entropy-cond} for a precise statement).
Thus, the $X$-nonnegativity of $f$ can be decided by convex \emph{relative entropy programming.}
The $X$-nonnegative signomials with at most one negative coefficient are called \textit{$X$-AGE}, and the signomials which decompose into a sum of such functions are called \textit{$X$-SAGE}. 
The recognition problem for $X$-SAGE signomials can likewise be decided by relative entropy programming.

The purpose of this article is to undertake the first structural analysis of the cones of $X$-SAGE signomials on exponents $\cA$, which we henceforth denote by $C_X(\cA)$.
At the outset of this research, our goals were to find counterparts to the many convex-combinatorial properties known for the unconstrained case $C_{\R^n}(\cA)$ \cite{FdW-2019,Katthaen:Naumann:Theobald:UnifiedFramework,mcw-2018}, and to understand conditional SAGE relative to techniques such as nonnegative circuit polynomials \cite{iliman-dewolff-resmathsci,Pantea2012,reznick-1989}.
Towards this end we have introduced an analysis tool of \emph{sublinear circuits} which we call \textit{the $X$-circuits of $\cA$}.
Our definition of these $X$-circuits (see Section \ref{sec:x_circuts_def}) centers on a local, orthant-wise, strict-sublinearity condition for the support function of $X$ composed with $\cA$.
This construction ensures that the special case of $\R^n$-circuits reduces to the simplicial circuits of the affine-linear matroid induced by $\cA$.

We demonstrate that analysis by $X$-circuits is extremely effective in describing many structural aspects of $X$-SAGE cones.
Our techniques are sufficiently robust that one can prove nearly every result in this manuscript assuming nothing of $X$ beyond convexity.
Some special treatment is given to the case when $X$ is polyhedral, as this reveals some striking interactions between discrete, convex, and so-called \textit{geometrically convex}  or \textit{multiplicatively convex} geometry (see Section \ref{sec:x_circuits_sage}).
In a broader sense, a selection of our results have consequences for numerical optimization, such as basis identification in optimization with SAGE certificates, and a procedure to simplify certain systems of power cone inequalities on the nonnegative orthant.

\subsection{Main contributions}\label{subsec:main_contrib}

We begin by introducing some limited notation.
The vector space of real $|\cA|$-tuples indexed by $\alpha\in\cA$ is denoted by $\R^{\cA}$.
The support function of a convex set $X$, denoted by $\sigma_X$, is the convex function defined by $\sigma_X(y) = \sup\{ y^T x : x \in X \}$.
We regard the exponent set $\cA \subset \R^n$ as a linear operator from $\R^{\cA}$ to $\R^n$ by $\cA \nu = \sum_{\alpha\in\cA}\alpha\nu_{\alpha}$; the corresponding adjoint is denoted $\cA^T$.
For concreteness, one might think of $\cA$ as a matrix with columns given by the exponents $\alpha$.
We continue to use $C_X(\cA)$ to denote the cone of $X$-SAGE signomials on $\cA$.
For each $\beta \in \cA$, we denote the corresponding cone of $X$-AGE functions by
\begin{equation}\label{eq:def_c_age}
    C_X(\cA,\beta) = \left\{ f \,:\, f =\sum_{\alpha\in\cA}c_{\alpha}\e^{\alpha} \text{ is } X\text{-nonnegative},~ c_{\setminus \beta} \geq \zerob \right\}
\end{equation}
where $c_{\setminus\beta}$ denotes the vector in $\R^{\cA \setminus\beta}$ formed by deleting $c_{\beta}$ from $c$.

The basic tools for our analysis are \textit{the $X$-circuits of $\cA$} (routinely abbreviated to \textit{$X$-circuits}).
We formulate the $X$-circuits of $\cA$ as nonzero vectors $\nu^\star \in \R^{\cA}$ at which the augmented support function $\nu\mapsto \sigma_X(-\cA \nu)$ exhibits a strict sublinearity condition (see Definition \ref{def:x_beta_circuit}).
We characterize $X$-circuits as generators of suitable convex cones in $\R^{\cA} \times \R$ and
usually focus on \textit{normalized} $X$-circuits $\lambda \in \R^{\cA}$, for which the nonnegative entries sum to unity.
Theorem \ref{thm:polyhedron_x_finite_circuits} shows that in the polyhedral case, $X$-circuits are exactly the generators of all one-dimensional elements of a suitable polyhedral fan.
A key consequence of Theorem \ref{thm:polyhedron_x_finite_circuits} is that when $X$ is a polyhedron, there are only finitely many normalized $X$-circuits.

Section \ref{sec:x_circuits_age_cones} uses the machinery of $X$-circuits to understand $X$-AGE cones.
First, we show that if a signomial generates an extreme ray of $C_X(\cA,\beta)$, then the dual variable $\nu$ which certifies its required relative entropy inequality must be an $X$-circuit (Theorem \ref{thm:x_circuits_age}).
Normalized $X$-circuits $\lambda$ are then associated to cones of \textit{$\lambda$-witnessed AGE functions} $C_X(\cA,\lambda)$.\footnote{When parsing $C_X(\cA,\beta)$ and $C_X(\cA,\lambda)$, the reader should note that $\beta$ and $\lambda$ live in different spaces.}
The functions in $C_X(\cA,\lambda)$ are $X$-nonnegative signomials admitting a nonnegativity certificate based on a damped power cone inequality in weights $\lambda$.
Theorem \ref{thm:primal_x_age_powercone} shows that every $X$-SAGE function can be written as a sum of $\lambda$-witnessed AGE functions for $X$-circuits $\lambda$.
In proving this, we formalize the connection between conditional SAGE and prior works for global nonnegativity \cite{iliman-dewolff-resmathsci,Pantea2012,reznick-1989}.
Theorem \ref{thm:primal_x_age_powercone} also motivates a basis identification technique where an approximate relative entropy certificate of $f \in C_X(\cA)$ may be refined by power cone programming.
Combining Theorems \ref{thm:polyhedron_x_finite_circuits} and \ref{thm:primal_x_age_powercone} yields a corollary that when $X$ is a polyhedron, cones of $X$-SAGE signomials are (in principle) power cone representable; this generalizes results by several authors in the unconstrained case \cite{averkov-2019,NaumannTheobald2020,papp-2019,wang2019}.

Section \ref{sec:x_circuits_sage} undertakes a thorough analysis of $C_X(\cA)$.
We begin by associating $X$-circuits $\lambda$ with affine functions $\phi_{\lambda} : \R^{\cA} \to \R$ given by $\phi_{\lambda}(y) = \sum_{\alpha\in\cA}y_{\alpha}\lambda_{\alpha} + \sigma_X(-\cA \lambda)$.
We define the \textit{circuit-generated cone} $G_X(\cA)$ as the smallest convex cone containing these functions and the constant function $y \mapsto 1$.
Upon embedding the affine functions on $\R^{\cA}$ into $\R^{\cA} \times \R$, Theorem \ref{thm:log_dual_x_sage_valid} provides the following identity between the dual SAGE cone $C_X(\cA)^*$ and the dual circuit-generated cone $G_X(\cA)^*$
\[
    C_X(\cA)^* = \cl\{ \exp y \,:\, (y,1) \in G_X(\cA)^* \}.
\]
\noindent Qualitatively, Theorem \ref{thm:log_dual_x_sage_valid} says $C_X(\cA)^*$ is not only convex in the classical sense, but also convex under a logarithmic transformation $S \mapsto \log S = \{ y : \exp y \in S \}$.
The property of a set being convex under this logarithmic transformation is known by various names, including \textit{log convexity} \cite{Agrawal2019}, \textit{geometric convexity} \cite{JarczykMatowski2002,Ozdemir2014}, or \textit{multiplicative convexity} \cite{Niculescu2000}.
This property has previously been considered in the literature on ordinary SAGE certificates (i.e., SAGE certificates for the special case $X=\R^n$) \cite{Katthaen:Naumann:Theobald:UnifiedFramework,mcw-2018}, but never in such a systematic way as in our analysis.
For example, in view of Theorem \ref{thm:log_dual_x_sage_valid} it becomes natural to consider $\Lambda_X^\star(\cA)$ -- \textit{the reduced $X$-circuits of $\cA$} -- as the normalized circuits $\lambda$ for which $\phi_{\lambda}$ generates an extreme ray of the circuit-generated cone.
The property of a circuit being ``reduced'' in this sense is highly restrictive, and yet (by Theorem \ref{thm:reduced_x_circuits}) we can construct $C_X(\cA)$ using only $\lambda$-witnessed AGE cones as $\lambda$ runs over $\Lambda_X^\star(\cA)$.
Finally, through a technical lemma (\ref{lem:exponentiate_separating_hyperplane}), we show how separating hyperplanes in the space of the dual circuit-generated cone may be mapped to separating hyperplanes in the exponentiated space of the dual SAGE cone.
This lemma has general applications in simplifying systems of certain power cone constraints on the nonnegative orthant; in our context, it serves as the basis for Theorem \ref{thm:polyhedra_reduced_x_circuits}, paraphrased below.
\begin{quote}
    \textit{If $X$ is a polyhedron and $C_X(\cA)$ consists of more than just posynomials, then} 
    \begin{equation*}
        C_X(\cA) = \sum_{\lambda \in \Lambda^\star_X(\cA)}  C_X(\cA,\lambda).
    \end{equation*}
    \textit{Moreover, there is no subset $\Lambda \subsetneq \Lambda_X^\star(\cA)$ for which $C_X(\cA) = \sum_{\lambda \in \Lambda} C_X(\cA,\lambda)$.}
\end{quote}
\noindent Theorem \ref{thm:polyhedra_reduced_x_circuits} provides the most efficient possible description of $C_X(\cA)$ in terms of power cone inequalities.
Its computational implications are addressed briefly in Section \ref{se:discussion}.

Throughout the article we illustrate key concepts with the half-line $X = [0,\infty)$.
Specifically, Example \ref{ex:circuits-univariate1} addresses the $[0,\infty)$-circuits of a generic point set $\cA \subset \R$, and Example \ref{ex:reduced_extreme_rays} covers the corresponding \textit{reduced} $[0,\infty)$-circuits.
This culminates with a complete characterization of the extreme rays of $C_X(\cA)$ for $X = [0,\infty)$ and $\cA \subset \R$ (Proposition \ref{prop:extremalsupportsconicdim1_alt}).

\subsection{Related work}\label{subsec:related_work}

Let us begin by introducing some basic concepts from discrete geometry.
The \textit{circuits} of the affine-linear matroid induced by $\cA$ are the nonzero vectors $\nu^\star \in \ker \cA \subset \R^{\cA}$ 
whose entries sum to zero, and whose supports are inclusion minimal among all vectors in $\ker \cA$ that sum to zero.
In the SAGE literature one is interested in \textit{simplicial circuits}.
These are the circuits $\nu^\star$ that, upon scaling by a suitable constant, have exactly one negative component.
The name \textit{simplicial} is used here because the convex hull of the support $\supp\nu^\star := \{\alpha \,:\, \nu^{\star}_{\alpha} \neq 0 \}$ forms a simplex (possibly of low dimension); exactly one element in $\supp\nu^\star$ is contained in the relative interior of this simplex.
These simplicial circuits are uniquely determined (up to scaling) by their supports.
It is therefore common to call a subset $A \subset \cA$ a simplicial circuit if its convex hull forms a simplex and has a relative interior containing exactly one element of $A$.

To situate conditional SAGE in the literature one should look to the close relatives of ordinary SAGE: the \textit{agiforms} of Reznick \cite{reznick-1989}, the \textit{monomial dominating posynomials} of Pantea, Koeppl and Craciun \cite{Pantea2012}, and the \textit{Sums-of-Nonnegative-Circuit} (SONC) polynomials of Iliman and de Wolff \cite{iliman-dewolff-resmathsci}.
The latter two works determined necessary and sufficient conditions for $\R^n_{+}$ and $\R^n$-nonnegativity of polynomials supported on a simplicial circuit, based on power cone inequalities in the polynomial's coefficients and circuit vector.
In our context, key developments in this area include Wang's discovery of conditions under which a SONC decomposition exists for a given polynomial \cite{wang-nonneg}, and Murray, Chandrasekaran, and Wierman's proof that the cone of SONC polynomials can be represented by a projection of a cone of SAGE signomials \cite[\S 5]{mcw-2018}.
From these results it is now understood that SONC and ordinary SAGE are equivalent to one another for purposes of certain structural analyses.
Our results show that the ``circuit number'' approach of SONC does not generalize to the $X$-nonnegativity problem in the same manner as SAGE.
However, it is possible to describe conditional SAGE in a way which is aesthetically similar to SONC via our $\lambda$-witnessed AGE cones.

To appreciate the structural results proven for $C_X(\cA)$ in this work, it is useful to mention some analogous results proven in the case $X = \R^n$.
As a signomial generalization of an earlier result by Reznick \cite{reznick-1989},
Murray, Chandrasekaran, and Wierman have shown that every signomial which generates an extreme ray of $C_{\R^n}(\cA)$ is supported on either a singleton or a simplicial circuit \cite{mcw-2018}.
Curiously, a given signomial $f$ can be extremal in $C_{\R^n}(\cA)$ for $\cA$ as the support of $f$, and yet nonextremal in $C_{\R^n}(\cA')$ for $\cA' \supsetneq \cA$.
To account for this, Katth\"an, Naumann, and Theobald introduced the concept of a \textit{reduced circuit}, which they used to obtain a complete characterization of the extreme rays of $C_{\R^n}(\cA)$  \cite{Katthaen:Naumann:Theobald:UnifiedFramework}.
Subsequently, Forsg{\aa}rd and de Wolff employed regular subdivisions, $A$-discriminants and tropical geometry to study how circuits affect the algebraic boundary of the signomial SAGE cone  \cite{FdW-2019}.
Our results include direct extensions of the above results by Murray et al.\ and Katth\"an et al.\ to the case of $X \subsetneq \R^n$.
For Forsg{\aa}rd and de Wolff's work, our 
\textit{circuit-generated cone} generalizes their \textit{Reznick cone}.

Now we turn to how SAGE can be used for optimization.
Given a signomial objective $f$ and a convex feasible set $X$, we have $\sup\{ \gamma \in \R : f - \gamma \text{ is } X\text{-SAGE}\} \leq \inf_{x \in X}f(x)$.
This procedure has been extended to a convex relaxation hierarchy for which A.\ Wang et al. have proven a completeness result \cite{wjyp-2020} (see also \cite{dickinson-povh}).
Very recently, additional SAGE-based hierarchies have been developed to approach a signomial's minimum from both above and below, including in the presence of nonconvex constraints \cite{dressler-murray-2021}.
Such techniques can be implemented using the \texttt{sageopt} python package and a reliable exponential cone solver such as MOSEK \cite{mosek,sageopt}.

On the polynomial optimization side, Karaca et al.\ developed a combined SAGE and Sums-of-Squares approach to optimization over (subsets of) the  nonnegative orthant \cite{karaca-2017}.
By consideration to the close SAGE-SONC relationship, one finds connections to works of Dressler et al.\ on polynomial optimization with SONC \cite{didW-2017,DKdW}.
As an alternative to SONC, one may work directly with a notion of \textit{SAGE polynomials} \cite[\S 5.1]{mcw-2018}.
The concept of SAGE polynomials is important because the corresponding nonnegativity certificates can be computed efficiently, and because they are transparently generalized to \textit{$X$-SAGE polynomials} \cite[\S 4]{mcw-2019}.
Our signomial results may be applied to conditional SAGE polynomials, however care must be taken in mapping between the two types of functions; see for example \cite[Theorems 1 and 2]{mcw-2019}.

\subsection{Some definitions and conventions}

Our terminology and notation for convex analysis is generally chosen to match that of Rockafellar \cite{rockafellar-book}.
Here we define terms and notation which are less commonly used or which differ from those of \cite{rockafellar-book}; additional standard definitions are reproduced in the appendix.
We abbreviate the line segment connecting $x$ and $y$ in $\R^n$ by $[x,y] = \{ \lambda x + (1-\lambda)y: 0 \leq \lambda \leq 1\}$.
A convex cone $K \subset \R^n$ is \emph{pointed} if it contains no lines.
A vector $v$ in a convex cone $K$ is called an \emph{edge generator} if $\{\lambda v \,:\, \lambda \geq 0\}$ is an extreme ray of $K$.
The \textit{polar} of a convex cone $K$ is $K^\circ = -K^*$, where $K^*$ is the dual cone to $K$.
The \textit{induced cone} of a convex set $S \subset \R^n$ is $\indco(S) \coloneqq \cl\{(s,\mu) \,:\, \mu > 0,\, s/\mu \in S\} \subset \R^{n+1}$, and the \textit{recession cone} is $\rec(S) \coloneqq \{ t \,:\, \exists s \in S \text{ such that } s + \lambda t \in S \; \forall\, \lambda \geq 0 \}$.

All logarithms are base-$e$, where $e$ is Euler's number.
We extend the scalar exponential function ``$\exp$'' to real vectors in an elementwise fashion.
The zero vector and vector of all ones (in appropriate spaces) are denoted $\zerob$ and $\oneb$ respectively.
The standard basis for $\R^{\cA}$ is denoted $\{\delta_{\alpha}\}_{\alpha\in\cA}$, and the support of a vector $c \in \R^{\cA}$ is $\supp c = \{\alpha:c_{\alpha} \neq 0 \}$.

\smallskip

\section*{Acknowledgements}

This work would not have been possible without an invitation from Bernd Sturmfels for R.M. to visit The Max Planck Institute for Mathematics in the Sciences (Leipzig, Germany) in late 2019.
R.M. was supported by an NSF Graduate Research Fellowship, and
T.T. was supported by DFG grant TH 1333/7-1. We thank the anonymous referees for their constructive feedback.

\section{Preliminaries}\label{sec:prelims}

Throughout this article, $X \subset \R^n$ is closed, convex, and nonempty, and the set $\cA \subset \R^n$ is nonempty and finite.
We only consider data $(\cA,X)$ where the functions $\{ \e^{\alpha} \}_{\alpha \in \cA}$ are linearly independent on $X$.
The purpose of this linear independence assumption is to ensure the $X$-nonnegativity cone does not contain a lineality space; equivalently, the assumption ensures the moment cone $\cone\{ \exp(\cA^T x) \in \R^{\cA} \,:\, x \in X\}$ is full-dimensional.

\begin{definition}\label{def:c_sage}
The \emph{$X$-SAGE cone with respect to the support $\mathcal{A}$}
is the Minkowski sum
\[
  C_X(\cA) \ = \ \sum_{\beta \in \cA} C_X(\cA,\beta),
\]
where $C_X(\cA,\beta)$ are the $X$-AGE cones defined in \eqref{eq:def_c_age}.
\end{definition}

\begin{remark}
    In this definition, all the signomials in the decomposition of the right hand
    side are also restricted to the support $\cA$. This is no loss of generality, 
    since any signomial $f$ on $\cA$, which is contained in
    $\sum_{\beta \in \mathcal{A'}} C_X(\cA',\beta)$ for some superset 
    $\mathcal{A}'$ of $\cA$, is also contained in 
    $\sum_{\beta \in \cA} C_X(\cA,\beta)$,
     see \cite[Corollary~1]{mcw-2019}.
\end{remark}

By adopting Definition \ref{def:c_sage}, it is clear that the problem of representing $C_X(\cA)$ reduces to the problem of representing the cones $C_X(\cA,\beta)$.
To state the representation of these cones we use the \textit{relative entropy function}
\[
    \relentr(\nu,c) = \sum_{\alpha \in \cA} \nu_\alpha \log\left(\frac{\nu_\alpha}{c_\alpha}\right).
\]
We use standard conventions where relative entropy is continuously extended to $\R^\cA_+ \times \R^\cA_+$, and define $D(\nu,c) = \infty$ if either $\nu$ or $c$ has a negative component.

\begin{proposition}[Theorem 1 of \cite{mcw-2019}]
\label{pr:rel-entropy-cond}
    A signomial $f = \sum_{\alpha\in \cA } c_\alpha \e^{\alpha}$ belongs to $C_X(\cA,\beta)$ if and only if there exists a vector $\nu \in \R^\cA$ that satisfies
    \begin{equation}
        \oneb^T\nu = 0 \quad \text{and} \quad \sigma_X(-\cA \nu) + \relentr(\nu_{\setminus \beta},e c_{\setminus \beta}) \leq c_\beta, \label{eq:relativentropycondition}
    \end{equation}
    where again, $\sigma_X(y) = \sup\{\, y^T x\, :\, x \in X\}$ for $y\in\R^n$. Such a vector $\nu$ is called a \emph{relative entropy certificate} for $f$.
\end{proposition}

Proposition \ref{pr:rel-entropy-cond} is important for computational optimization.
For example, if $X$ is the unit ball in the Euclidean norm, then $\sigma_X(-\cA \nu) = \| \cA \nu\|_2$, and so \eqref{eq:relativentropycondition} becomes a mixed relative-entropy and second-order-cone inequality.
More generally, the formulation is tractable whenever we can efficiently represent the epigraph of the support function of $X$.

Since the relative entropy condition in Proposition~\ref{pr:rel-entropy-cond} is essential
for our treatment, we outline its proof. Adopt $I_X$ as the indicator function
of $X$, with $I_X(x) = 0$ for $x\in X$ and $I_X(x) = \infty$ otherwise.
Given $f = \sum_{\alpha\in\cA}c_{\alpha}\e^{\alpha}$ with $c_{\setminus \beta} \geq \zerob$,
the primal formulation for $X$-nonnegativity of $f$ is
\begin{equation}\label{eq:nonnegativity_as_primal_min}
   \inf_{\substack{x \in \R^n \\ t \in \R^{\cA}}}\left\{\ I_X(x) + \sum_{\alpha\in\cA \setminus \beta} c_{\alpha}\exp{t_{\alpha}} ~:~ t_{\alpha} = (\alpha-\beta)^T x \; \, \forall \, \alpha\in\cA \right\}  \geq -c_{\beta}.
\end{equation}
The formulation \eqref{eq:relativentropycondition} is simply the dual to \eqref{eq:nonnegativity_as_primal_min} using the machinery of convex conjugate functions.
In particular, the relative entropy certificate $\nu$ in \eqref{eq:relativentropycondition} is the dual variable to the equality constraints in \eqref{eq:nonnegativity_as_primal_min}.

The larger goal of this article is to reveal additional structure in the $X$-SAGE cones $C_X(\cA)$ that is not immediately apparent from Proposition \ref{pr:rel-entropy-cond}.
From the case $X = \R^n$, the additional structure concerned the supports of signomials that generate extreme rays of $C_{\R^n}(\cA,\beta)$ or $C_{\R^n}(\cA)$.
In this context it is standard to use the term \textit{simplicial circuit} in the sense of subsets $A \subset \cA$.
Specifically, $A \subset \cA$ is a simplicial circuit if it is a minimal affinely dependent set and $\conv A$ has $|A|-1$ extreme points.
This definition of circuits in terms of these subsets $A \subset \cA$ is equivalent to the definition involving numeric vectors $\nu^\star \in \R^{\cA}$; see \cite{FdW-2019}.

\begin{proposition}[Theorem 5 of \cite{mcw-2018}]\label{pr:rn-sage-via-circuits}
    Let $\beta \in \cA$.
    A signomial $f = \sum_{\alpha\in \cA } c_\alpha \e^{\alpha}$ belongs to 
    $C_{\R^n}(\cA,\beta)$ if and only if it can be written as a finite sum $f = \sum_{i=1}^k f^{(i)}$ of signomials
    \[
     f^{(i)} \ = \ \sum_{\alpha\in \cA } c^{(i)}_\alpha \e^{\alpha} \in C_{\R^n}(\cA,\beta),
    \quad 1 \le i \le k,
    \]
    such that the supports $\{ \alpha \in \cA: c^{(i)}_{\alpha} \neq 0 \}$ are either singletons or simplicial circuits.
\end{proposition}

Of course, in view of Definition 2.1, Proposition \ref{pr:rn-sage-via-circuits} tells us every $f \in C_{\R^n}(\cA)$ similarly decomposes into AGE functions supported on singletons and simplicial circuits.

Revealing the full structure of conditional SAGE cones requires consideration to more than just a signomial's support.
Therefore, thinking in terms of affine-linear circuits as subsets $A \subset \cA$ will not suit our purposes.
The following definition codifies our convention of considering affine-linear circuits as numeric vectors.

\begin{definition}\label{de:rn-circuit}
    A nonzero vector 
    $\nu^\star \in \{\nu \in \R^{\cA} \, : \, \oneb^T \nu = 0 \}$
    in the kernel of the linear operator 
    $\nu \mapsto \cA \nu = \sum_{\alpha \in \cA} \alpha \nu_{\alpha}$
    is called an \emph{$\R^n$-circuit} if it is minimally supported and has exactly one negative component.
\end{definition}

It is possible that a given $\cA$ has no $\R^n$-circuits, but then every $\alpha \in \cA$ would be an extreme point of $\conv\cA$.
This is a degenerate case that results in $C_{\R^n}(\cA)$ containing only posynomials, but we still give consideration to this possibility throughout the article.
In the language of Definition \ref{de:rn-circuit}, we combine Propositions~\ref{pr:rel-entropy-cond} 
and~\ref{pr:rn-sage-via-circuits} to obtain the following formulation.

\begin{proposition}[Theorem 4.4 of \cite{FdW-2019}] Let $\beta \in \cA$.
    A signomial $f = \sum_{\alpha\in \cA } c_\alpha \e^{\alpha}$ belongs to 
    $C_{\R^n}(\cA,\beta)$ if and only if there exist $k \ge 0$ and signomials
    $
    f^{(i)} \ = \ \sum_{\alpha\in \cA } c^{(i)}_\alpha \e^{\alpha} \in C_{\R}(\cA,\beta),
    \quad 1 \le i \le k,
    $
    with $f = \sum_{i=1}^k f^{(i)}$ and such that for any signomial $f^{(i)}$ which is not
    supported on a singleton, there exists an $\R^n$-circuit $\nu^{(i)} \in \R^\cA$ with 
    $\relentr(\nu^{(i)}_{\setminus \beta},e c^{(i)}_{\setminus \beta}) \leq c^{(i)}_\beta.$
\end{proposition}

\section{Sublinear circuits induced by a point set}\label{sec:x_circuts_def}

We begin this section with a functional analytic definition for the $X$-circuits of a point set $\cA$, generalizing $\R^n$-circuits to a constrained setting.
After revealing various elementary properties and discussing some examples,
we characterize $X$-circuits in more geometric terms in 
Theorems \ref{thm:characterize_circuits} and \ref{thm:polyhedron_x_finite_circuits}.
In particular the latter theorem interprets $X$-circuits in terms of normal fans
when $X$ is a polyhedron.
In Example \ref{ex:circuits-univariate1}, we determine the $[0,\infty)$-circuits of 
a univariate support set $\cA \subset \R$; the example is developed further in 
Section \ref{sec:x_circuits_sage} and culminates in a theorem completely
characterizing the extreme rays of the resulting $X$-SAGE cone 
$C_{[0,\infty)}(\cA)$ in Section \ref{sec:primal_perspective}.

The derivations in this section are purely combinatorial and convex-geometric, and make no mention of signomials.
However, the 
definition of $X$-circuits is ultimately chosen to prepare for studying $X$-SAGE cones, and in particular it relates to distinguished vectors $\nu \in \R^{\cA}$ that might
satisfy \eqref{eq:relativentropycondition} for certain $c \in \R^{\cA}$.
Note that \eqref{eq:relativentropycondition} has an implicit constraint $\nu_{\setminus\beta} \geq \zerob$ arising from our extended-real-valued definition of relative entropy.
To avoid dependence on relative entropy in this section, we frame our discussion of $X$-circuits in terms of cones
\begin{equation}\label{eq:nbeta}
  N_{\beta} = \{ \nu \in \R^{\cA} \,:\, \nu_{\setminus \beta} \geq \zerob,~ \oneb^T \nu = 0 \}
\end{equation}
for vectors $\beta \in \cA$.

\begin{definition}\label{def:x_beta_circuit}
    A vector $\nu^\star \in N_\beta$ is an \textit{$X$-circuit of $\cA$} (or simply, an \textit{$X$-circuit}) if (1) it is nonzero, (2) $\sigma_X(-\cA \nu^\star) < \infty$, and (3) it cannot be written as a convex combination of two non-proportional $\nu^{(1)},\nu^{(2)} \in N_\beta$, for which $\nu \mapsto \sigma_X(-\cA \nu)$ is linear on $[\nu^{(1)},\nu^{(2)}]$.
\end{definition}

The third condition is equivalent to strict sublinearity of $\nu \mapsto \sigma_X(-\cA \nu)$ on any line segment in $N_{\beta}$ that contains $\nu^\star$, except for the trivial line segments which generate a single ray.
The central importance of the sublinearity condition leads us to refer to $X$-circuits also
as \textit{sublinear circuits}; the latter term is helpful in remembering the definition early in our development.

\begin{remark}
    In the special case $X=\R^n$, condition (2) simplifies to $\cA \nu = \zerob$.
    In conjunction with the definition of $N_{\beta}$, this shows that the 
    special case $X=\R^n$ of Definition~\ref{def:x_beta_circuit} matches exactly
    with Definition~\ref{de:rn-circuit} of $\R^n$-circuits.
\end{remark}

Conceptually, Definition \ref{def:x_beta_circuit} indicates that $X$-circuits are essential in capturing the behavior of the augmented support function $\nu \mapsto \sigma_X(-\cA \nu)$ on the given $N_{\beta}$.
While developing this concept formally
it is convenient for us to enumerate $\nu^+ \coloneqq \{ \alpha \,:\, \nu_\alpha > 0 \}$, and to identify the unique index $\nu^- \coloneqq \beta \in \cA$ where $\nu_\beta < 0$.
Note that positive homogeneity of the support function tells us that the property of being a sublinear circuit is invariant under scaling by positive constants.
A sublinear circuit is \textit{normalized} if its unique negative term $\nu_{\beta}$ has $\nu_{\beta} = -1$, in which case we usually denote it by the symbol $\lambda$
rather than $\nu$.
We can normalize a given sublinear circuit by taking the ratio with its infinity norm $\lambda = \nu / \| \nu \|_{\infty}$, because $\|\nu\|_{\infty}=|\nu_\beta|$ for all vectors $\nu \in N_{\beta}$.

\begin{example}\emph{(The conic case.)}
    It is straightforward to determine which $\nu \in N_\beta$ are $X$-circuits of $\cA$ when $X$ is a cone.
    In such a setting, the support function of $X$ can only take on the values zero and positive infinity. Hence, $\nu \mapsto \sigma_X(-\cA \nu)$ is trivially linear over all of $V_\beta \coloneqq \{ \nu \in N_\beta \,:\, \sigma_X(-\cA \nu) < \infty \}$.
    Notice that $V_{\beta}$ is a cone and that $\sigma_X(-\cA\nu) = 0$ may be reformulated as $\nu \in (\cA^T X)^*$.
    Standard conic duality calculations (see Proposition \ref{prop:convex_analysis:dual_of_linear_image}) show that $(\cA^T X)^* = \ker \cA + \cA^\dagger X^*$, where $\cA^\dagger$ denotes the Moore-Penrose pseudo-inverse of $\cA$.
    Thus
    \[
    V_\beta = (\ker \cA + \cA^\dagger X^*) \cap N_\beta
    \]
    and the $X$-circuits $\nu \in N_{\beta}$ are precisely the edge generators of $V_{\beta}$.
    
    Regarding again the special case $X=\R^n$ from this conic perspective,
    we have $X^* = \{\zerob\}$, so $\cA^{\dagger} X^* = \{\zerob\}$, and $\ker \cA + \cA^{\dagger} X^* = \ker \cA$, which implies $V_\beta = \ker \cA \cap N_\beta$.
        It is easily shown that edge generators of $\ker \cA \cap N_\beta$ are precisely those $\nu \in \ker\cA \cap N_\beta \setminus\{\zerob\}$ for which $\nu^+ = \{ \alpha \,:\, \nu_\alpha > 0\}$ are affinely independent, which recovers the
    matroid-theoretic notion of affine-linear simplicial circuits from the point of view of subsets $A \subset \cA$.
\end{example}

The following proposition shows that the affine-independence property is a necessary condition for all sublinear circuits.
The proposition provides insight because it shows an $X$-circuit $\nu$ with $X \subset \R^n$ is restricted to $|\supp \nu| \leq n+2$.

\begin{proposition}\label{prop:x_circuit_aff_indep}
    If $\nu^\star \in N_\beta$ is an $X$-circuit, then $(\nu^\star)^+ = \supp \nu^\star \setminus \beta$ is affinely independent.
\end{proposition}
\begin{proof}
    From a fixed $\nu^\star \in N_\beta$ construct $z = -\cA \nu^\star$ and $U = \{ \nu \in N_\beta \,:\, -\cA \nu = z,\, \nu_\beta = \nu^\star_\beta \}$.
    The function $\nu \mapsto \sigma_X(-\cA \nu)$ is a constant and equal to $\sigma_X(z)$ on $U$, and so in order for $\nu^\star$ to be an $X$-circuit, it must be a vertex of the polytope $U$.
    The set $U$ is in 1-to-1 correspondence with $W=\{ w \in \R^{\cA \setminus \beta}_{+} \,:\, \sum_{\alpha \in \cA \setminus \beta}(\beta-\alpha) w_\alpha = z,\, \mathds{1}^T w = -\nu^\star_\beta \}$ by identifying $w = \nu_{\setminus \beta}$. 
    In matrix notation, we can write $W = \{ w \in \R^{\cA \setminus\beta}_+ \,:\, M w = (z,-\nu_{\beta}^\star)\}$ by forming the matrix $M$ with columns $\{ (\beta-\alpha, 1) \}_{\alpha\in\cA\setminus\beta}$ indexed by $\alpha\in\cA\setminus\beta$.
   
    Basic polyhedral geometry tells us that all vertices $w^\star$ of $W$ use an affinely independent set of columns from $M$.
    Furthermore, a given set of columns from $M$ is affinely independent if and only if the corresponding indices of the columns (as vectors $\alpha\in\cA\setminus\beta$) are affinely independent.
    Since the correspondence between $\nu \in U$ and $w \in W$ preserves extremality, the vertices of $U$ have affinely independent positive support $\nu^+$.
\end{proof}

The converse of Proposition \ref{prop:x_circuit_aff_indep} is not true.
This is to say: not every vector $\nu \in N_{\beta}$ with affinely independent $\nu^+$ is an $X$-circuit.

\begin{example}\label{ex:aff_indep_not_sufficient}
    Let $\cA \subset \R^2$ contain $\alpha_1 = (0, 0)$, $\alpha_2= (1,0)$, and $\alpha_3 = (0,1)$, and consider $X = \{ x \in \R^2 \ : \ x \ge u\}$ for some fixed point $u \in \R^2$.
    The vector $\nu^\star = (-2,1,1)$ has $(\nu^\star)^- = \alpha_1 = (0,0)$, and $(\nu^\star)^+
    = \{\alpha_2,\alpha_3\} = \{(1,0),(0,1)\}$ is affinely independent.
    Considering $\nu^{(1)} = (-2,2,0)$ and $\nu^{(2)} = (-2,0,2)$, we have $\nu^\star =\frac{1}{2}(\nu^{(1)}+\nu^{(2)}) \in \relint L$ for $L \coloneqq [\nu^{(1)},\nu^{(2)}]$.
    Moreover, the mapping $\nu \mapsto \sigma_X(-\cA \nu)$
    is linear on $L$, because for any $\mu_1, \mu_2 \ge 0$ with
    $\mu_1+\mu_2 = 1$ we have 
    \begin{align*}
       \sigma_X(\cA (-\mu_1 \nu^{(1)} - \mu_2 \nu^{(2)}))
      & = \sigma_X((-2\mu_1,-2\mu_2)) = -2\mu_1 u_1- 2\mu_2 u_2 \\
      & =
      \sigma_X((-2\mu_1,0)) + \sigma_X((0,-2\mu_2)).
    \end{align*}
    The last equality is true
    since $(1,1)$ maximizes both the objective functions
    $x \mapsto (-2\mu_1,0)^T x$ and $x \mapsto (0,-2\mu_2)^T x$ on $X$.
\end{example}

With the basic exercise of Example \ref{ex:aff_indep_not_sufficient} complete, we turn to characterizing sublinear circuits in full generality.

\begin{theorem}\label{thm:characterize_circuits}
    Fix $\beta \in \cA$.
    The convex cone generated by
    \[
    T = \{\, (\nu, \sigma_X(-\cA \nu)) \,:\, \nu \in N_\beta,\, \sigma_X(-\cA \nu) < \infty\}
    \]
    is pointed (i.e., it contains no lines) and closed.
    A vector $\nu^\star \in N_\beta$ is an $X$-circuit of $\cA$ if and only if $(\nu^\star, \sigma_X(-\cA \nu^\star))$ is an edge generator for $\cone T$.
\end{theorem}

\begin{proof}
    Let $Q$ denote the closed convex set $Q = \{ \nu \,:\, \nu \in N_\beta,\, \sigma_X(-\cA \nu) < \infty\}$.
    The claim of the theorem is trivially true if $Q = \{\zerob\}$, in which case there are no $X$-circuits $\nu \in N_\beta$ and $\cone T = \{(\zerob,0)\}$ has no extreme rays.
    We therefore assume for the duration of the proof that $Q$ contains a nonzero vector.
    
    We turn to showing $\cone T$ is closed and pointed, particularly beginning with pointedness.
    For this, observe $\cone T \subset N_\beta \times \R$.
    Since $N_\beta$ contains no lines, there are no lines in $\cone T$ of the form $(\nu,\tau)$ with $\nu \neq \zerob$.
    Meanwhile, we know that the line spanned by $(\zerob,1)$ cannot be contained in $\cone T$, since $\sigma_X(-\cA \zerob) = 0$.
    Now we turn to closedness of $\cone T$.
    Since $Q$ is contained within $ N_\beta$, we may normalize $Q$ against $\{ \nu \,:\, \nu_\beta = -1\}$:
    $Q = \cone Q_1$ for the nonempty compact convex set
    $Q_1 \coloneqq \{ \lambda \,:\, \lambda \in Q, \lambda_\beta = -1\}$.
    From $Q_1$ we construct
    $T_1 = \{ (\lambda,\sigma_X(-\cA \lambda)) \,:\, \lambda \in Q_1 \}$.
    The set $T_1$ inherits compactness from $Q_1$ (by continuity of $\lambda \mapsto \sigma_X(-\cA \lambda)$), and the convex hull $T_2 = \conv T_1$ inherits compactness from $T_1$ (as the convex hull of a compact set is compact).
    It is evident that $T_2$ does not contain the zero vector, and so by \cite[Corollary 9.6.1]{rockafellar-book} we have that $\cone T_2$ is closed.
    We finish this phase of the proof by identifying $\cone T = \cone T_2$.

    At this point we have that $\cone T$ is the convex hull of its extreme rays; it remains to determine the nature of these extreme rays.
    Since $T$ is a generating set for $\cone T$ and contains only vectors of the form $(\nu,\sigma_X(-\cA \nu))$, every edge generator of $\cone T$ is given by a nonzero vector $(\nu^\star,\sigma_X(-\cA \nu^\star))$ for appropriate $\nu^\star$.
    It is clear that $\nu^\star$ must be an $X$-circuit in order for $(\nu^\star,\sigma_X(-\cA \nu^\star))$ to be an edge generator of $\cone T$.
    The harder direction is to show that $\nu^\star$ being an $X$-circuit is \textit{sufficient} for $(\nu^\star,\sigma_X(-\cA\nu^\star))$ to be an edge generator for $\cone T$.
    
    To handle this direction, begin by defining an affinely independent set $\mathcal{V} = \{ \nu^{(i)} \}_{i=1}^\ell$ and a vector $\theta$ in the relative interior of  $\Delta_\ell \coloneqq \{ z \in \R^\ell_+ \,:\, \mathds{1}^T z = 1 \}$, where $\nu^\star = \sum_{i=1}^\ell \theta_i \nu^{(i)}$ and
    \[
        \sigma_X(-\cA \nu^\star) = \textstyle\sum_{i=1}^\ell \theta_i \sigma_X(-\cA \nu^{(i)}).
    \]
    We claim that $\nu \mapsto \sigma_X(-\cA \nu)$ is linear on the entirety of $\conv \mathcal{V}$.
    To see why, note that the assumption on $\nu^\star$ relative to $\mathcal{V}$ means the elements of $\Phi \coloneqq \{ (\nu^{(i)},\sigma_X(-\cA \nu^{(i)})) \,:\, i \in [\ell]\cup\{\star\} \}$ lie on a common hyperplane on the boundary of the epigraph $H = \{ (\nu,t) \,:\, \sigma_X(-\cA \nu) \leq t \}$.
    Since $\nu \mapsto \sigma_X(-\cA \nu)$ is convex, $H$ is a convex set, and there is some proper face $F$ of $H$ which contains $\Phi$.
    It is evident that $\nu \mapsto \sigma_X(-\cA \nu)$ is linear on the projection of that face $\hat{F} = \{ \nu\,:\, \exists t \in \R \, \; (\nu,t) \in F \}$.
    Since $\conv\mathcal{V} \subset \hat{F}$, this proves our claim regarding linearity of $\nu \mapsto \sigma_X(-\cA \nu)$ on $\conv \mathcal{V}$.
    
    By the above argument: if $\nu^\star$ is an $X$-circuit, then for every $\theta \in \relint \Delta_\ell$ and affinely independent $\mathcal{V} = \{ \nu^{(i)} \}_{i=1}^\ell \subset N_\beta$ with $\cone \mathcal{V} \neq \cone \{ \nu^\star \}$, we have
    \[
    (\nu^\star,\sigma_X(-\cA \nu^\star)) \neq \textstyle\sum_{i=1}^\ell \theta_i \left(\nu^{(i)},\sigma_X(-\cA \nu^{(i)})\right).
    \]
    From Carath\'eodory's Theorem, restricting to affinely independent $\mathcal{V} \subset T$ is sufficient to test extremality in $\cone T$.
    Therefore, every circuit $\nu^\star \in N_\beta$ induces an edge generator for $\cone T$.
\end{proof}

When considering the set ``$T$'' in Theorem \ref{thm:characterize_circuits}, it is natural to expect that for polyhedral $X$ there are only finitely many extreme rays in the cone $\cone T$, and hence only finitely many normalized $X$-circuits.
The remainder of this section serves to prove this fact; here we use the concept of \textit{normal fans} from polyhedral geometry. See, e.g., \cite[Chapter~7]{Ziegler}
(for the bounded case of polytopes), \cite[Section~5.4]{gkz-1994} or 
\cite[Chapter~2]{sturmfels-gbcp}.
For each face $F$ of a polyhedron $P$, there is an associated \textit{outer normal cone}
\[
    \Nsf_P(F) = \{ w \,:\, z^T w = \sigma_P(w) ~\forall~z\in F\}.
\]
Clearly, the support function of a polyhedron $P$ is linear on every outer
normal cone,
and in particular the linear representation may be given by $\sigma_P(w) = z^T w$ 
for any $z \in F$.
We obtain the \textit{outer normal fan} of $P$ by collecting all outer normal cones:
\[
    \mathcal{O}(P) = \{ \Nsf_P(F) \,:\, F \text{ is a face of } P \}
\]
The support of $\mathcal{O}(P)$ is the polar $\rec(P)^{\circ}$.
The full-dimensional linearity domains of the support function are the outer normal cones
of the vertices of $P$ (see also \cite[Section~1]{fillastre-izmestiev-2017}).

\begin{theorem}\label{thm:polyhedron_x_finite_circuits}
    If $X$ is polyhedral, then $\nu \in N_\beta \setminus \{\zerob\}$ is an 
    $X$-circuit if and only if
     $\cone\{\nu\}$ is a ray in $\mathcal{O}(-\cA^T X + N_{\beta}^\circ)$.
    Consequently, a polyhedral set $X$ has finitely many normalized circuits.
\end{theorem}
\begin{proof}
    Let $P = -\cA^T X + N_\beta^\circ$. Using the characterization in 
    \cite[Theorem 14.2]{rockafellar-book}, the polar of its recession cone
    can be expressed as
    \[
     (\rec P)^\circ = \{ \nu \,:\, \sigma_X(-\cA \nu) < \infty\} \cap N_\beta,
    \]
    where we have also used the property 
    $\sigma_X(-\cA \nu) = \sup_{x \in X} (-\cA \nu)^T x
    = \sup_{x \in -\cA^T X} \nu^T x 
    = \sigma_{-\cA^T X}(\nu).
    $
    In particular, this also gives $\sigma_X(-\cA \nu) = \sigma_P(\nu)$.
    From $P$ construct the outer normal fan $\mathcal{O} \coloneqq \mathcal{O}(P)$.
    We claim that $\cone\{\nu\}$ is a ray in $\mathcal{O}$.
    
    It is clear that if a cone $K \in \mathcal{O}$ is associated to a face $F$ of $P$, then we
    may express $\sigma_P(\nu) = z^T \nu$ for any $z \in F$, and so $\sigma_P(\nu) \equiv \sigma_X(-\cA\nu)$ is linear on $K$.
    Since the support of $\mathcal{O}$ is $\rec(P)^{\circ}$, the
    cones $K \in \mathcal{O}$ partition $(\rec P)^\circ$, i.e.,
    \begin{equation*}
        (\rec P)^\circ  =  \bigcup_{K \in \mathcal{O}} \relint(K),
    \end{equation*}
    and if $K,K'$ are distinct elements in $\mathcal{O}$, then $\relint K \cap \relint K' = \emptyset$.
    Therefore, every $\nu \in N_\beta \setminus \{\zerob\}$ for which $\sigma_X(-\cA \nu) < \infty$ is associated with a unique $K \in \mathcal{O}$, by way of $\nu \in \relint K$.
    
    Fix $\nu \in (\rec P)^\circ$, and let $K$ be the associated element of $\mathcal{O}$ that contains $\nu$ in its relative interior.
    If $K$ is of dimension greater than 1, $\nu$ can be expressed as a convex combination of non-proportional $\nu^{(1)},\nu^{(2)} \in K$ -- and clearly $\bar{\nu} \mapsto \sigma_X(-\cA \bar{\nu}) \equiv \sigma_P(\bar{\nu})$ would be linear on the interval $[\nu^{(1)},\nu^{(2)}]$.
    Thus for $\nu$ to be an $X$-circuit, it is necessary that $K$ be of dimension 1.
    Since $P$ is a polyhedron, $\mathcal{O}$ is induced by finitely many faces.
    Thus there are finitely many $K \in \mathcal{O}$ with $\dim K = 1$ and in turn finitely many normalized $X$-circuits of $\cA$.
    
    Conversely, let $\nu^\star \in N_\beta \setminus \{\zerob\}$ and $\cone\{\nu^\star\}$ be a ray in 
    $\mathcal{O}$. Since $\mathcal{O}$
    is supported on $\rec(P)^{\circ}$,
    we have $\sigma_X(- \cA \nu^\star) = \sigma_P(\nu^\star) < \infty$.
    
    Let $\nu^{(1)}, \nu^{(2)} \in N_{\beta}$ be non-proportional and 
    $\tau \in (0,1)$ satisfy 
    $\nu^\star = \tau \nu^{(1)} + (1-\tau) \nu^{(2)}$. If $\nu^{(1)}$ or $\nu^{(2)}$
    is outside of $\rec(P)^{\circ}$, say, $\nu^{(1)}$, then
    $\sigma_X(-\cA \nu^{(1)}) = \infty$ and thus the mapping
    $\nu \mapsto \sigma_X(-\cA \nu)$ cannot be linear on
    $[\nu^{(1)}, \nu^{(2)}]$. Hence, we can assume that 
    $\nu^{(1)}, \nu^{(2)} \in \rec(P)^{\circ}$.
    
    We have to show that the mapping
    \[
      g \, : [0,1] \to \R, \quad \theta \mapsto \sigma_P(\theta \nu^{(1)} + (1-\theta) \nu^{(2)})
    \]
    is not linear.
    
    Consider the restriction of the fan $\mathcal{O}$ to the cone
    $C := \cone \{\nu^{(1)}, \nu^{(2)} \}$, that is, the collection of all the 
    cones in $\{\Nsf _P(F) \cap C \, : \, F \text{ is a face of }P\}$. 
    This is a fan $\mathcal{O}'$ supported on the two-dimensional cone
    $S := \rec(P)^{\circ} \cap C$.
    On the set $S$, we consider the restricted mapping 
    $(\sigma_P)|_{S} \, : \, S \to \R$, $w \mapsto \sigma_P(w)$. 
    The linearity domains of $(\sigma_P)|_S$ are the two-dimensional cones
    in $\mathcal{O}'$. Since $\cone\{v^\star\}$ is
    a ray in the fan $\mathcal{O}$ and thus also in the fan 
    $\mathcal{O}'$, the vectors $\nu^{(1)}$ and
    $\nu^{(2)}$ are contained in different two-dimensional cones of the fan
    $\mathcal{O}'$. Hence, the mapping $g$ is not linear.
    Altogether, this shows that $\nu^\star$ is an $X$-circuit.
\end{proof}

\begin{example}
\label{ex:circuits-univariate1}
    We consider as a running example
    the one-dimensional case of $X = [0,\infty)$ and
    $\mathcal{A} = \{\alpha_1, \ldots, \alpha_m \} \subset \R$
    where we can assume $\alpha_1 < \cdots < \alpha_m$.
    In this running example we index by integers $i \in [m] := \{1,\ldots,m\}$ rather than by elements $\alpha\in\cA$.
    Therefore we identify $\R^{\cA}$ with $\R^m$ and use $\delta_i$ for the $i^{\text{th}}$ unit vector in $\R^m$ (for each $i \in [m]$).
    Under these conventions, $\cA$ is regarded as a row vector in $\R^{1 \times m}$ and $\cA^T = (\alpha_1,\ldots,\alpha_m)$ is a column vector in $\R^{m}$.
    We claim that
    the normalized $X$-circuits $\lambda \in \R^m$ are the vectors either of the form (1) $\lambda = \delta_k - \delta_j$ for $j < k$ or of the form (2)
    \[
    \lambda = \left(\frac{\alpha_j - \alpha_i}{\alpha_k - \alpha_i}\right)\delta_{k} + \left(\frac{\alpha_k-\alpha_j}{\alpha_k - \alpha_i}\right)\delta_i - \delta_j \quad\text{ for }\quad i < j < k.
    \]
    Note that vectors of type (2) satisfy $\cA \lambda = 0$, and in fact are the unique such vectors that also satisfy $\supp \lambda = \{i,j,k\}$, $\lambda_j = -1$, $\lambda_i,\lambda_k > 0$, $\mathds{1}^T \lambda = 0$.
    
    To derive this claim we consider for fixed $j \in [m]$ the polyhedron
    $P = -\mathcal{A}^T X + N_{j}^{\circ}$ from 
    Theorem~\ref{thm:polyhedron_x_finite_circuits}.
    It is evident that this polyhedron is a cone, that may be expressed as
    \[
      P \ = \ \cone\{ (-\alpha_1, \ldots, -\alpha_m)\} 
        + \R \cdot \mathds{1}  -
        \sum_{\ell \in [m] \setminus j} \cone \{\delta_{\ell}\}.
    \]
    The rays of its normal fan are the extreme rays of its polar
    \begin{equation}
      \label{eq:circuits-univariate1}
      P^{\circ} = (\rec P)^{\circ} = \{ \nu \in \R^{m} \, : \,
      (-\alpha_1, \ldots, -\alpha_m)^T \nu \le 0, \; \ \mathds{1}^T \nu = 0, \;
      \nu_{\ell} \ge 0 \text{ for } \ell \in [m] \setminus j \}.
    \end{equation}
    Note here that this gives us exactly the set ``$Q$'' from the proof of Theorem \ref{thm:characterize_circuits}. This happens because $X$ is conic and hence the support function $\sigma_X(-\cA\nu)$ evaluates to zero for every $X$-circuit $\nu$.
    By Proposition~\ref{prop:x_circuit_aff_indep}, each $X$-circuit in $N_{j}$
    has at most three non-vanishing components $\nu_i, \nu_j, \nu_k$,
    and, moreover, it has $m-2$ of the inequalities 
    in~\eqref{eq:circuits-univariate1} binding. If all those binding inequalities are
    of the form~$\nu_{\ell} \ge 0$, then with $\sigma_X(-\cA\nu)<\infty$,  we obtain the normalized 
    $X$-circuits of $\cA$ of type~(1). Now assume that the 
    inequality $(-\alpha_1, \ldots, -\alpha_m)^T \nu \le 0$ is binding for some 
    normalized $X$-circuit $\nu$ of $\cA$. Since
    the sign pattern $(-,+,+)$ for $(\nu_i,\nu_j,\nu_k)$ in conjunction with $\mathds{1}^T \nu = 0$ leads to $(-\alpha_1, \ldots, -\alpha_m)^T \nu < 0$, and 
    the sign pattern
    $(+,+,-)$ contradicts the $X$-circuit condition $\sigma_X(-\cA\nu)<\infty$, we obtain
    the normalized $X$-circuits of $\cA$ of type~(2).
\end{example}

For the example classes of the nonnegative orthant
and the cube $[-1,1]^n$, we refer the reader to
\cite{naumann-theobald-2021}.

\section{Sublinear circuits in AGE cones}\label{sec:x_circuits_age_cones}

In this section, we show how the $X$-AGE cones $C_X(\mathcal{A},\beta)$
can be further decomposed using sublinear circuits.
These decompositions lay the foundation to understand the extreme rays of the 
conditional SAGE cone $C_X(\cA)$.
Our first result here is a necessary criterion for an $X$-AGE function $f$ to be extremal in 
$C_X(\mathcal{A},\beta)$, which states that all of its relative entropy certificates must be $X$-circuits
(see Theorem~\ref{thm:x_circuits_age}). Definition \ref{def:lambda_witnessed_age_cone} introduces 
\textit{$\lambda$-witnessed AGE cones} as the subset of signomials in $C_{X}(\mathcal{A},\beta)$ 
whose nonnegativity is certified by a given normalized vector~$\lambda$. 
Theorem~\ref{thm:primal_x_age_powercone} then decomposes $C_X(\mathcal{A},\beta)$ through
the $\lambda$-witnessed AGE cones, where $\lambda$ is a normalized $X$-circuit.
As a consequence, for polyhedral $X$, the cone $C_X(\mathcal{A},\beta)$ is power-cone representable (see Corollary \ref{cor:primal_x_age_powercone_polyhedra}).

In the last part of this section we prove two propositions on explicit representations for primal 
and dual $\lambda$-witnessed AGE cones. Proposition \ref{prop:weighted_dual_x_age_powercone} 
in particular is very important for a characterization of dual $X$-SAGE cones, as it reveals a multiplicative convexity property used extensively in Section \ref{sec:x_circuits_sage}.

The following lemma provides a construction to decompose an $X$-AGE function into simpler summands, under a local linearity condition on the support function $\nu \mapsto \sigma_X(-\cA \nu)$.

\begin{lemma}\label{lem:linear_suppfunc_age_decomp}
    Let $f = \sum_{\alpha\in \cA } c_\alpha \e^{\alpha}$ be $X$-AGE with negative term $c_\beta < 0$.
    If $\nu$ is a relative entropy certificate for $f$ which
    can be written as a convex combination $\nu = \sum_{i=1}^k \theta_i \nu^{(i)}$ of non-proportional $\nu^{(i)} \in N_{\beta}$ and $\tilde{\nu} \mapsto \sigma_X(-\cA \tilde{\nu})$ is linear on $\conv\{\nu^{(i)}\}_{i=1}^k$, then $f$ is not extremal in $C_X(\cA,\beta)$.
\end{lemma}
\begin{proof}
    Construct vectors $c^{(i)}$ by
    \begin{equation}
        c^{(i)}_\alpha = \begin{cases} (c_\alpha/\nu_\alpha) \nu^{(i)}_\alpha &\text{ if } \alpha \in \nu^+ \\
        0 & \text{ otherwise } \end{cases}
        \qquad \text{ for all } \alpha \in \cA \setminus \beta,\label{eq:XcircuitAGEvectors}
    \end{equation}
    and $c^{(i)}_\beta = \sigma_X(-\cA \nu^{(i)}) + \relentr(\nu^{(i)}_{\setminus \beta}, e c^{(i)}_{\setminus \beta})$.
    These $c^{(i)}$ define $X$-AGE signomials by construction, and they inherit non-proportionality from the $\nu^{(i)}$.
    We need to show that $\sum_{i=1}^k \theta_i c^{(i)} \leq c$, which will establish that $f$ can be decomposed as a sum of these non-proportional $X$-AGE functions (possibly with an added posynomial).
    
    For indices $\alpha \in \nu^+$, the construction \eqref{eq:XcircuitAGEvectors} relative to $\nu$ and $\{\nu^{(i)}\}_{i=1}^k$ actually ensures $\sum_{i=1}^k \theta_i c^{(i)}_\alpha = c_\alpha$.
    For indices $\alpha \in \supp c \setminus \supp \nu$ we have $\sum_{i=1}^k \theta_i c^{(i)}_\alpha = 0 \leq c_\alpha$.
    The definitions of $\nu^{(i)}$ ensure
    \begin{equation}
    \sigma_X(-\cA \nu) = \sigma_X\left(-\cA(\textstyle\sum_{i=1}^k \theta_i \nu^{(i)})\right) = \textstyle\sum_{i=1}^k \theta_i \sigma_X(-\cA \nu^{(i)}).
    \label{eq:linearity_benefit_x_age}
    \end{equation}
    Meanwhile, \eqref{eq:XcircuitAGEvectors} provides $\nu^{(i)}_\alpha/c^{(i)}_\alpha = \nu_\alpha / c_\alpha$, which may be combined with $\sum_{i=1}^k \theta_i \nu^{(i)}_\alpha = \nu_\alpha\; \forall\; \alpha \in \cA$ to deduce 
    \begin{equation}
     \sum_{i=1}^k \theta_i \relentr(\nu^{(i)}_{\setminus \beta}, e c^{(i)}_{\setminus \beta}) = \relentr(\nu_{\setminus \beta},e c_{\setminus \beta}). \label{eq:circuit_x_age_relent_linear}
    \end{equation}
    We combine \eqref{eq:linearity_benefit_x_age} and \eqref{eq:circuit_x_age_relent_linear} to obtain the desired result
    \[
    \sum_{i=1}^k \theta_i c^{(i)}_\beta = \sum_{i=1}^k \theta_i\left(\sigma_X(-\cA \nu^{(i)}) + \relentr(\nu^{(i)}_{\setminus \beta},e c^{(i)}_{\setminus \beta})\right) = \sigma_X(-\cA \nu) + \relentr(\nu_{\setminus \beta},e c_{\setminus \beta}) \leq c_{\beta}.
    \]
\end{proof}

\begin{theorem}\label{thm:x_circuits_age}
    Let $f = \sum_{\alpha\in \cA } c_\alpha \e^{\alpha}$ be $X$-AGE with negative term $c_\beta < 0$.
If $f$ has a relative entropy certificate which is not an $X$-circuit, then $f$ is not extremal in $C_X(\cA,\beta)$.
\end{theorem}
\begin{proof}
    If $f$ is an $X$-AGE function with $c_\beta < 0$ and $\nu$ satisfies \eqref{eq:relativentropycondition}, then we must have $\nu \neq \zerob$ and $\sigma_X(-\cA \nu) < \infty$.
    By the definition of an $X$-circuit, $\nu$ may be written as a convex combination $\nu = \theta \nu^{(1)} + (1-\theta)\nu^{(2)}$ where $\bar{\nu} \mapsto \sigma_X(-\cA \bar{\nu})$ is linear on $[\nu^{(1)},\nu^{(2)}]$, and furthermore the $\nu^{(i)}$ are not proportional.
    We can therefore invoke Lemma \ref{lem:linear_suppfunc_age_decomp} to prove the claim.
\end{proof}

In the remainder of this section we eliminate the degree of freedom associated with $\nu$ laying on a ray.
For each $\beta \in \cA$, we introduce the following notation for the associated set of normalized $X$-circuits of $\cA$
\begin{equation*}
    \Lambda_X(\cA,\beta) = \{ \lambda \in N_\beta \,:\, \lambda \text{ is an } X\text{-circuit of } \cA ,\, \lambda_\beta = -1\}.\label{eq:normalized_x_beta_circuits}
\end{equation*}
The set of all normalized $X$-circuits of $\cA$ is denoted $\Lambda_X(\cA)$.
The main reason for introducing this notation is how it interacts with the following definition.

\begin{definition}\label{def:lambda_witnessed_age_cone}
Given a vector $\lambda \in N_{\beta}$ with $\lambda_{\beta} = -1$, the \textit{$\lambda$-witnessed AGE cone} is
\begin{equation}\label{eq:weighted_prim_x_age_powercone}
    C_X(\cA,\lambda) = \left\{~ \sum_{\alpha\in\cA}c_{\alpha}\e^{\alpha} \,:\, \prod_{\alpha \in \lambda^+} \left[\frac{c_\alpha}{\lambda_\alpha} \right]^{\lambda_\alpha} \geq - c_\beta \exp\left(\sigma_X(-\cA \lambda)\right) ,~ c_{\setminus \beta} \geq \zerob \right\}.
\end{equation}
\end{definition}
We show below that every signomial in $C_X(\cA,\lambda)$ is nonnegative on $X$.
The term ``witnessed'' in ``$\lambda$-witnessed AGE cone'' is chosen to reflect the defining role of $\lambda$ in the nonnegativity certificate.
We only use $\lambda$-witnessed AGE cones for theoretical purposes, and only with $\lambda \in \Lambda_X(\cA)$.
Possible computational uses (particularly with $\lambda \not\in \Lambda_X(\cA)$) are offered in Section \ref{se:discussion}.

\begin{theorem}\label{thm:primal_x_age_powercone}
   Let $\Lambda_X(\cA)\ne\emptyset$. The cone $C_X(\cA,\beta)$ can be written as the convex hull of $\lambda$-witnessed AGE cones,
    where $\lambda$ runs over the normalized $X$-circuits, that is,
    \[
     C_X(\cA,\beta) = \conv\bigcup_{\lambda \in \Lambda_X(\cA,\beta)} C_X(\cA,\lambda).
    \]
\end{theorem}

Note here that for any $\beta \in \mathcal{A}$ and (normalized)  $\lambda \in N_{\beta}$, we have $\mathbb{R}^{\mathcal{A}}_+ \subset C_X(\mathcal{A},\lambda)$.

\begin{proof}
    Theorem \ref{thm:x_circuits_age} already tells us that for $\Lambda_X(\cA)\ne\emptyset$, $C_X(\cA,\beta)$ may be expressed as the convex hull of $X$-AGE functions $f = \sum_{\alpha \in \cA} c_\alpha \e^{\alpha}$ 
    which have $X$-circuits as relative entropy certificates.
    Therefore it suffices to show that (i) for any such function, the normalized $X$-circuit
    $\lambda = \nu / |\nu_{\beta}|$
    is such that $(c,\lambda)$ satisfy the condition in \eqref{eq:weighted_prim_x_age_powercone}, and (ii) if \textit{any} $(c,\lambda)$ satisfy \eqref{eq:weighted_prim_x_age_powercone}, then the resulting signomial is nonnegative on $X$.
    We will actually do both of these in one step.
    
    Suppose $\nu \in N_\beta$ is restricted to satisfy $\nu = s \lambda$ for a variable $s \geq 0$ and a fixed $\lambda \in \Lambda_X(\cA,\beta)$.
    It suffices to show that the set of $c \in \R^{\cA}$ for which
    \[
    \exists s \geq  0 \,:\, \nu = s \lambda \text{ and } \sigma_X(-\cA \nu) + \relentr(\nu_{\setminus \beta},e c_{\setminus \beta}) \leq c_\beta
    \]
    is the same as \eqref{eq:weighted_prim_x_age_powercone}.
    
    Let $r(\nu) = \sigma_X(-\cA \nu) + \relentr(\nu_{\setminus\beta},e c_{\setminus\beta})$.
    Apply positive homogeneity of the support function to see $\sigma_X(-\cA \nu) = |\nu_{\beta}| \sigma_X(\cA \nu /|\nu_{\beta}|)$, and use $\nu = s \lambda$ to infer $s = |\nu_{\beta}|$ and $\sigma_X(-\cA \nu / |\nu_{\beta}|) = \sigma_X(-\cA \lambda)$.
    Abbreviate $d := \sigma_X(-\cA \lambda)$ and substitute $\sum_{\alpha\in\lambda^+}\nu_{\alpha} = |\nu_{\beta}|$ to obtain
    \[
    r(\nu) = \textstyle\sum_{\alpha \in \lambda^+}\left( \nu_\alpha \log(\nu_\alpha / c_\alpha) - \nu_\alpha + \nu_\alpha d\right).
    \]
    The term $d$ may be moved into the logarithm by identifying $\nu_\alpha d = \nu_\alpha \log(1 / \exp(-d))$.
    For $\alpha \in \lambda^+$ we define scaled terms $\tilde{c}_\alpha = c_\alpha \exp(-d)$, so that $r(\nu) = \textstyle\sum_{\alpha \in \lambda^+} \nu_\alpha \log(\nu_\alpha / \tilde{c}_\alpha) - \nu_\alpha$.
    By Proposition \ref{prop:misc:powercone_relentr}, there exists a $\nu = s \lambda$ for which $r(\nu) \leq c_\beta$ if and only if
    \begin{equation}
      \label{eq:lambda-witnessed-end}
    -c_\beta \leq \prod_{\alpha \in \lambda^+} [\tilde{c}_\alpha / \lambda_\alpha]^{\lambda_\alpha}.
    \end{equation}
    Since $[\tilde{c}_\alpha / \lambda_\alpha]^{\lambda_\alpha} = [c_\alpha / \lambda_\alpha]^{\lambda_\alpha} \left( \exp(-d) \right)^{\lambda_\alpha}$ and $\prod_{\alpha \in \lambda^+} \left(  \exp(-d) \right)^{\lambda_\alpha} = \exp(-d)$,
    \eqref{eq:lambda-witnessed-end} can be recognized as the inequality occurring
    within~\eqref{eq:weighted_prim_x_age_powercone}, which completes the proof.
\end{proof}

Theorem \ref{thm:primal_x_age_powercone} shows how $\lambda$-witnessed AGE cones provide a window to the structure of \textit{full} AGE cones $C_X(\cA,\beta)$.
To appreciate the benefit of this perspective, it is necessary to consider the more elementary ``power cone.''
In our context, the primal power cone associated with a normalized $X$-circuit $\lambda \in \R^{\cA}$ is 
\[
    \powercone(\lambda) = \{ z \in \R^{\supp \lambda} \,:\, \textstyle\prod_{\alpha \in \lambda^+} z_{\alpha}^{\lambda_\alpha} \geq |z_\beta|,~ z_{\setminus \beta} \geq \zerob,~ \beta \coloneqq \lambda^{-} \};
\]
the corresponding dual cone is given by
\[
    \powercone(\lambda)^* = \{ w \in \R^{\supp \lambda} \,:\, \textstyle\prod_{\alpha \in \lambda^+}[w_\alpha / \lambda_\alpha]^{\lambda_\alpha} \geq |w_\beta|,~ w_{\setminus \beta} \geq \zerob,~ \beta \coloneqq \lambda^- \}. 
\]
It should be evident that $C_X(\cA,\lambda)$ can be formulated in terms of a dual $\lambda$-weighted power cone; a precise formula is provided momentarily.
For now we give a corollary concerning power cone representability and second-order representability
of $C_X(\cA)$ when $X$ is a polyhedron (see \cite{averkov-2019,BenTal2001} for formal definitions).

\begin{corollary}\label{cor:primal_x_age_powercone_polyhedra}
    If $X$ is a polyhedron, then $C_X(\cA)$ is power cone representable.
    If in addition $\cA^T X$ is \textit{rational}, then $C_X(\cA)$ is second-order representable
    and thus has semidefinite extension degree~2.
\end{corollary}
\begin{proof}
    We can assume  $\Lambda_X(\cA)\neq\emptyset$, since otherwise $C_X(\cA)=\R_+^\cA$ and the claim follows. By Theorem \ref{thm:polyhedron_x_finite_circuits}, polyhedral $X$ have finitely many $X$-circuits, up to scaling.
    Apply Theorem \ref{thm:x_circuits_age} and finiteness of the normalized circuits $\Lambda_X(\cA)$ to write
    \[
    C_X(\cA) = \sum_{\lambda \in \Lambda_X(\cA)} C_X(\cA,\lambda).
    \]
    The first claim follows as each of the finitely many sets $C_X(\cA,\lambda)$ appearing in the above sum are (dual) power cone representable.
    For the second claim observe that under the rationality assumptions we have $\Lambda_X(\cA) \subset \mathbb{Q}^{\cA}$.
    Using $\beta \coloneqq \lambda^-$ and $m \coloneqq |\supp \lambda|$, it is known that the $m$-dimensional $\lambda$-weighted power cone (and its dual) are second-order representable when $\lambda_{\setminus \beta}$ is a rational vector in the $(m-1)$-dimensional probability simplex \cite[Section 3.4]{BenTal2001}.
    The last claim follows as the semidefinite extension degree of the second-order cone is 
    two \cite[Section 2.3]{BenTal2001}.
\end{proof}

The first part of Corollary \ref{cor:primal_x_age_powercone_polyhedra} generalizes the case $X=\R^n$ considered by Papp for polynomials \cite{papp-2019}.
That aspect of the corollary has uses in computational optimization when applied judiciously.
The second part of Corollary \ref{cor:primal_x_age_powercone_polyhedra} generalizes results by Averkov \cite{averkov-2019} and Wang and Magron \cite{wang2019} for ordinary SAGE polynomials, and recent results by Naumann and Theobald for several types of ordinary SAGE-like certificates \cite{NaumannTheobald2020}.
We have deliberately framed the second part of the corollary in abstract terms (semidefinite extension degree), because that aspect of the corollary seems not useful for computational optimization.

We now work towards finding a simple representation of dual $\lambda$-witnessed AGE cones $C_X(\cA,\lambda)^*$.
We begin this process by regarding the primal as a cone of coefficients contained in $\R^{\cA}$, and finding an explicit representation of the primal in terms of the elementary dual power cone  $\powercone(\lambda)^*$.
Towards that end we introduce a diagonal linear operator $S_\lambda : \R^{\cA} \to \R^{\supp \lambda}$ where $(S_\lambda w)_\alpha = w_\alpha$ for $\alpha \in \lambda^+$, and $(S_\lambda w)_\beta = w_\beta \exp(\sigma_X(-\cA \lambda))$ for $\beta \coloneqq \lambda^-$.
Recall that $\delta_\beta \in \R^{\cA}$ denotes the standard basis vector corresponding to $\beta \in \cA$, i.e., $\delta_\beta^T w = w_\beta$ for $w \in \R^\cA$.

\begin{proposition}\label{prop:weighted_age_basic_powercone}
    For $\lambda \in N_{\beta}$ with $\lambda_{\beta} = -1$ and $\sigma_X(-\cA \lambda) < \infty$, the $\lambda$-witnessed AGE cone admits the representation
    \begin{equation}
        C_X(\cA,\lambda) = \{ c \in \R^{\cA} \,:\, \beta \coloneqq \lambda^-,~ c_{\setminus \beta} \geq \zerob,~ (S_\lambda c - r \delta_\beta) \in \powercone(\lambda)^*,~ r \geq 0 \}.\label{eq:weighted_age_basic_powercone}
    \end{equation}
\end{proposition}
\begin{proof}
    First, we note that some inequality constraints $c_{\setminus \beta} \geq \zerob$ are implied by $(S_\lambda c - r \delta_\beta) \in \powercone(\lambda)^*$.
    It is necessary to include the inequality constraints explicitly, to account for the case when $\supp \lambda \subsetneq \cA$.
    The condition $(S_\lambda c - r \delta_\beta) \in \powercone(\lambda)^*$ can be rewritten as
    \begin{equation}
        \prod_{\alpha \in \lambda^+} [c_\alpha / \lambda_\alpha]^{\lambda_\alpha} \geq | c_\beta \exp(\sigma_X(-\cA \lambda)) - r|.\label{eq:prove_basic_age_powercone_lift}
    \end{equation}
    Meanwhile, the minimum of $| c_\beta \exp(\sigma_X(-\cA \lambda)) - r|$ over $r \geq 0$ is attained at $r = 0$ when $c_\beta < 0$ and $r = c_{\beta}\exp(\sigma_X(-\cA\lambda))$ when $c_\beta \geq 0$.
    In the $c_\beta < 0$ case the constraint \eqref{eq:prove_basic_age_powercone_lift} becomes
    \[
        \prod_{\alpha \in \lambda^+} [c_\alpha / \lambda_\alpha]^{\lambda_\alpha} \geq -c_\beta \exp(\sigma_X(-\cA \lambda)).
    \]
    In the $c_\beta \geq 0$ case the constraint \eqref{eq:prove_basic_age_powercone_lift} is vacuous, since $\prod_{\alpha \in \lambda^+} [c_\alpha / \lambda_\alpha]^{\lambda_\alpha} \geq 0$ is implied by $c_{\setminus \beta} \geq \zerob$.
    As the constraint in the preceding display is similarly vacuous when $c_\beta > 0$, we see that it can be used in lieu of \eqref{eq:prove_basic_age_powercone_lift} without loss of generality.
\end{proof}

We can appeal to Proposition \ref{prop:weighted_age_basic_powercone} to find a representation for $C_X(\cA,\lambda)^*$ which is analogous to Equation (\ref{eq:weighted_prim_x_age_powercone}).
Again, the dual is computed by regarding the primal as a cone of \textit{coefficients}.

\begin{proposition}\label{prop:weighted_dual_x_age_powercone}
    For $\lambda \in N_{\beta}$ with $\lambda_{\beta} = -1$ and $\sigma_X(-\cA \lambda) < \infty$, the dual $\lambda$-witnessed AGE cone is given by 
    \begin{equation}\label{eq:weighted_dual_x_age_powercone}
    C_X(\cA,\lambda)^* = \left\{ v \in \R^{\cA}_+ \,:\, \beta \coloneqq \lambda^-,~ \exp(\sigma_X(-\cA \lambda)) \prod_{\alpha \in \lambda^+} v_\alpha^{\lambda_\alpha} \geq v_\beta\right\}.
    \end{equation}
\end{proposition}
\begin{proof}
    Let $\beta = \lambda^-$ as is usual.
    To $v \in \R^\cA$ associate $\mathrm{Val}(v) = \inf\{ v^T c \,:\, c \in C_X(\cA,\lambda)\}$.
    A vector $v$ belongs to $C_X(\cA,\lambda)^*$ if and only if $\mathrm{Val}(v) = 0$.
    We will find constraints on $v$ so that the dual feasible set for computing $\mathrm{Val}(v)$ is nonempty, which in turn will imply $\mathrm{Val}(v) = 0$.
    
    We begin by noting that for any element $\alpha \in \cA \setminus \supp \lambda$, the only constraints on $c_\alpha, v_\alpha$ for $c \in C_X(\cA,\lambda),v \in C_X(\cA,\lambda)^*$ are $c_\alpha \geq 0, v_\alpha \geq 0$;
    therefore we assume $\cA = \supp \lambda$ for the remainder of the proof.
    When considering the given expression for $\mathrm{Val}(v)$ as a primal problem, we compute a dual using \eqref{eq:weighted_age_basic_powercone} from Proposition \ref{prop:weighted_age_basic_powercone}.
    Under the assumption $\cA = \supp \lambda$, the constraint $c_{\setminus \beta} \geq \zerob$ is implied by $(S_\lambda c - r \delta_\beta) \in \powercone(\lambda)^*$.
    Therefore when forming a Lagrangian for $\mathrm{Val}(v)$ using \eqref{eq:weighted_age_basic_powercone}, the dual variable to ``$c_{\setminus \beta} \geq \zerob$'' may be omitted.
    
    For the remaining constraints $(S_\lambda c - r \delta_\beta) \in \powercone(\lambda)^*$ and $r \geq 0$ we use dual variables $\mu \in \powercone(\lambda)$ and $t \in \R_+$ respectively; the Lagrangian is
    \begin{align*}
    \mathcal{L}(c,r,\mu,t) 
        &= v^T c - \mu^T(S_\lambda c - r \delta_\beta) - tr \\
        &= c^T (v - S_\lambda^T \mu) - r(t - \mu_\beta).
    \end{align*}
    For the Lagrangian to be bounded below over $c \in \R^{\cA}$ and $r \in \R$, it is necessary and sufficient that $v = S_\lambda^\intercal \mu$ and $\mu_\beta = t$.
    Since we have assumed $\supp \lambda = \cA$ and $\sigma_X(-\cA \lambda) < \infty$, the diagonal linear operator $S_\lambda$ is symmetric positive definite, so we can express the requirements on $\mu, t$ as
    \[
        S_\lambda^{-1} v = \mu \quad \text{and} \quad \mu_\beta = t.
    \]
    Therefore the conditions $S_\lambda^{-1} v \in \powercone(\lambda),~ v_\beta \geq 0$ are equivalent to
    \begin{align*}
    \mathrm{Val}(v)
        &= \inf\bigg\{\sup\{\mathcal{L}(c,r,\mu,t) \,:\, (\mu,t) \in \powercone(\lambda) \times \R_+ \}~:~ (c,r) \in \R^{\cA} \times \R \bigg\} \\
        &= \sup\bigg\{\inf\{\mathcal{L}(c,r,\mu,t) \,:\, (c,r) \in \R^{\cA} \times \R \} ~:~ (\mu,t) \in \powercone(\lambda) \times \R_+ \bigg\} = 0.
    \end{align*}
    The proposition follows by applying the definitions of $\powercone(\lambda)$ and $S_\lambda$.
\end{proof}

\section{Reduced sublinear circuits in SAGE cones}\label{sec:x_circuits_sage}

The previous section showed that an $X$-SAGE cone is generated by $X$-circuits.
Here we seek a much sharper characterization: are all $X$-circuits really necessary?
The answer to this question depends on whether one means to reconstruct an individual AGE cone, or the larger SAGE cone.
For example, by reinterpreting results from \cite{mcw-2018}, we may infer that every simplicial $\R^n$-circuit $\lambda \in \Lambda_{\R^n}(\cA,\beta)$ generates a $\lambda$-witnessed AGE cone containing an extreme ray of $C_{\R^n}(\cA,\beta)$.
In this way, every $\R^n$-circuit is needed if one requires complete reconstruction of individual AGE cones.
However, Katth\"an, Naumann, and Theobald showed that many extreme rays of AGE cones are \textit{not} extreme when considered in the sum $C_{\R^n}(\cA) = \sum_{\beta \in \cA} C_{\R^n}(\cA,\beta)$.
Specifically, an $\R^n$-circuit $\lambda \in \Lambda_{\R^n}(\cA)$ is only needed in $C_{\R^n}(\cA)$ if exactly one element of $\cA$ hits the relative interior of $\conv(\supp \lambda)$ \cite[Proposition 4.4]{Katthaen:Naumann:Theobald:UnifiedFramework}.
Circuits satisfying this property were called \textit{reduced}.
The goal of this section is to develop a reducedness criterion for $X$-circuits that yields the most efficient construction of $C_X(\cA)$ by $\lambda$-witnessed AGE cones, see 
Theorems~\ref{thm:reduced_x_circuits} 
and~\ref{thm:polyhedra_reduced_x_circuits}.
Achieving this goal is more difficult than obtaining the results from earlier sections.
Therefore we begin by summarizing and discussing the results, and we provide proofs in later subsections.

\subsection{Definitions, results, and discussion}

The definition of a reduced $\R^n$-circuit is of a purely combinatorial nature, involving the circuit's support.
This is appropriate because when speaking of affine-linear simplicial circuits, the normalized vector representation $\lambda$ is completely determined by its support.
In the context of $X$-circuits, we no longer have this property.
Therefore when developing reduced \textit{$X$-circuits} it is useful to have a different characterization of reduced \textit{$\R^n$-circuits}.
Here we can consider how Forsg{\aa}rd and de Wolff defined the \textit{Reznick cone of $\cA$} as the conic hull $R_{\R^n}(\cA) \coloneqq \cone\Lambda_{\R^n}(\cA)$ and  -- in the language of Katth\"an et al. -- subsequently proved that an $\R^n$-circuit $\lambda$ is an edge generator of $R_{\R^n}(\cA)$ if and only if it is reduced \cite{FdW-2019}.

Our definition of reduced $X$-circuits involves edge generators of a certain cone in one higher dimension than the Reznick cone.
To describe the cone and facilitate later analysis, we need the following definition.

\begin{definition}\label{def:func_form}
   The \textit{functional form} of an $X$-circuit $\nu \in \R^{\cA}$ is $\phi_\nu : \R^\cA \to \R$ defined by
    \[
    \phi_\nu(y) = \sum_{\alpha \in \cA} y_\alpha \nu_\alpha + \sigma_X(-\cA \nu).
    \]
\end{definition}

We routinely overload notation and use $\phi_\nu = (\nu,\sigma_X(-\cA \nu)) \in \R^{\cA} \times \R$ to denote the functional form of a given $X$-circuit.
When representing the functional form of an $X$-circuit by a vector in $\R^{\cA} \times \R$, the scalar $\phi_\nu(y)$ can be expressed as an inner product $\phi_\nu(y) = (y,1)^T \phi_\nu$.

\begin{definition}\label{def:circuit_graph}
    The \textit{circuit-generated cone} (shortly, \textit{CG cone}) 
    of $(\cA,X)$ is
    \[
    G_X(\cA) = \cone\left(\{ \phi_\lambda \,:\, \lambda \in \Lambda_X(\cA)\}\cup\{(\zerob,1)\}\right),
    \]
    where $(\zerob,1) \in \R^{\cA} \times \R$.
\end{definition}

The idea of generating a cone from augmented circuit vectors $(\nu,\sigma_X(-\cA \nu)) \in \R^{\cA} \times \R$ clearly parallels Theorem \ref{thm:characterize_circuits}.
While the cones from Theorem \ref{thm:characterize_circuits} are considered for one $\beta\in\cA$ at a time, the CG cone accounts for all $X$-circuits at once.
The CG cone also includes an extra generator that ultimately serves to make the following definition more stringent.

\begin{definition}\label{def:reduced_circuits}
    The \textit{reduced $X$-circuits of $\cA$} are the vectors $\nu$ where $\nu/\|\nu\|_{\infty} \in \Lambda_X(\cA)$ and the corresponding functional form $\phi_\nu$ generates an extreme ray of $G_X(\cA)$. The set of \textit{normalized} reduced $X$-circuits is henceforth denoted $\Lambda_X^\star(\cA)$.
\end{definition}

There is a subtle issue here that in order for reduced $X$-circuits to be of any use to us, the CG cone must be pointed (else $G_X(\cA)$ would have no extreme rays whatsoever).
We show later in this section that our stated assumption of linear independence of $\{\e^{\alpha}\}_{\alpha\in\cA}$ on $X$ ensures $G_X(\cA)$ is pointed.
Regardless of whether or not the CG cone is pointed, we have the following theorem.

\begin{theorem}\label{thm:log_dual_x_sage_valid}
    $C_X(\cA)^* = \cl\{ \exp y \,:\, (y,1) \in G_X(\cA)^* \}$.
\end{theorem}

Theorem \ref{thm:log_dual_x_sage_valid} is noteworthy in several respects.
It demonstrates that $C_X(\cA)^*$ is convex in the usual sense and convex under a logarithmic transformation $S \mapsto \log S = \{ y : \exp y \in S \}$.
This second form of convexity is a significant structural property.
For example, if we know that the log of the moment cone $\cl(\cone\{\exp(\cA^T x) : x \in X \})$ is not convex, then it should be that $C_X(\cA)$ does not contain all $X$-nonnegative signomials on $\cA$.
Additionally, Theorem \ref{thm:log_dual_x_sage_valid} can be reverse-engineered to arrive at the concept of a reduced $X$-circuit: the definition is chosen so that $(y,1)$ belongs to $G_X(\cA)^*$ if and only if $\phi_{\lambda}(y) \geq 0$ for all $\lambda$ in $\Lambda_X^\star(\cA)$.
Here, Theorem \ref{thm:log_dual_x_sage_valid} is a tool that we combine with convex duality to obtain the following results.

\begin{theorem}\label{thm:reduced_x_circuits}
    If $\Lambda_X(\cA)$ is empty, then $C_X(\cA) = \R^{\cA}_+$. Otherwise,
    \begin{equation}\label{eq:reduced_x_circuit_thm_conclusion}
        C_X(\cA) = \cl\left(\conv \bigcup \big\{  C_X(\cA,\lambda)~:~ \lambda \in \Lambda_X^\star(\cA) \big\}\right).
    \end{equation}
\end{theorem}

We point out how Theorem \ref{thm:reduced_x_circuits} involves a closure around the union over $\lambda$-witnessed AGE cones, while Theorem \ref{thm:primal_x_age_powercone} has no such closure.
The need for the closure here stems from an application of an infinite version of conic duality in the course of the theorem's proof, while our proof of Theorem \ref{thm:primal_x_age_powercone} required no duality at all.
The requisite use of conic duality is simpler when $X$ is a polyhedron, as the following theorem suggests.

\begin{theorem}\label{thm:polyhedra_reduced_x_circuits}
    If $X$ is a polyhedron and $\Lambda_X(\cA)$ is nonempty, then the associated conditional SAGE cone is given by the finite Minkowski sum
    \begin{equation}\label{eq:polyhedra_reduced_x_circuit_thm_conclusion}
        C_X(\cA) = \sum_{\lambda \in \Lambda^\star_X(\cA)}  C_X(\cA,\lambda).
    \end{equation}
    Moreover, there is no proper subset $\Lambda \subsetneq \Lambda_X^\star(\cA)$ for which $C_X(\cA) = \sum_{\lambda \in \Lambda} C_X(\cA,\lambda)$.
\end{theorem}

The first part of Theorem \ref{thm:polyhedra_reduced_x_circuits} follows easily from the arguments we use to prove Theorem \ref{thm:reduced_x_circuits}.
The second part of the theorem is much more delicate, and in fact is the reason why $G_X(\cA)$ is defined in the manner of \ref{def:circuit_graph}, rather than merely $\cone\{ \phi_\lambda \,:\, \lambda \in \Lambda_X(\cA) \}$.

The task of actually finding the reduced $X$-circuits of $\cA$ is difficult.
When $X$ is a polyhedron there are finitely many such $X$-circuits, but the naive method for finding them involves Fourier-Motzkin elimination on a set of potentially very high dimension.
There is more hope for this problem when $X$ is a cone.
In that case, $X$-circuits are the extreme rays of $(\ker \cA + \cA^{\dagger} X^*) \cap N_{\beta}$ for $\beta\in\cA$, and no lifting is needed to find these extreme rays with a computer.
The reduced $X$-circuits could then be computed by finding the extreme rays of the convex cone generated by the $X$-circuits.
The following detailed example finds the reduced $X$-circuits of $\cA$ in the univariate case with $X = [0,\infty)$.
The claim made in the example is used in Section \ref{sec:primal_perspective}.

\begin{example}\label{ex:reduced_extreme_rays}
    We continue the running example of $X=[0,\infty)$ from Example~\ref{ex:circuits-univariate1}.
    In particular recall $\cA = \{\alpha_1,\ldots,\alpha_m\}$ for $\alpha_1 < \cdots < \alpha_m$, indexing by $i \in [m]$, and working with standard basis $\delta_i \in \R^m$.
    We claim that
    \begin{equation}\label{eq:Lambda_X_star_univariate_cone}
    \Lambda_{[0,\infty)}^\star(\cA) = \{ \delta_2 - \delta_1 \} \cup \Lambda_{\R}^\star(\cA)
    \end{equation}
    where we have the following formula from \cite[Prop. 4.4]{Katthaen:Naumann:Theobald:UnifiedFramework}
    \[
        \Lambda_{\R}^\star(\cA) = \left\{ \left( \frac{\alpha_{i+1}- \alpha_i}{\alpha_{i+1}-\alpha_{i-1}}\right)\delta_{i-1} + \left( \frac{\alpha_i- 
			\alpha_{i-1}}{\alpha_{i+1}-\alpha_{i-1}}\right)\delta_{i+1} - \delta_i \,:\,  ~ 1 < i < m \right\}.
    \]
    
    As a first step towards seeing this, observe that since $X=[0,\infty)$ is a cone, the functional form of a $[0,\infty)$-circuit $\nu$ is simply
    $\phi_{\nu}(y) = \sum_{i=1}^m y_{i} \nu_{i}$.
    Hence, the reduced $[0,\infty)$-circuits are exactly the edge generators of the cone $\cone\Lambda_{[0,\infty)}$ generated by all the $[0,\infty)$-circuits of types (1) and (2) listed in Example~\ref{ex:circuits-univariate1}.
    Therefore, we have to show that $\{\delta_2-\delta_1\}\cup\Lambda_{\R}^\star(\cA)$ are exactly the normalized edge generators of $\cone\Lambda_{[0,\infty)}$.
    
    For the $X$-circuits $\delta_j-  \delta_i$ ($j > i$) of type~(1) in Example~\ref{ex:circuits-univariate1}, we show
    they decompose if $j>i+1$ or $i>1$.
    For $j>i+1$, this is apparent from the decomposition
    \[
       \delta_{j} - \delta_{i} \ = \ (\delta_{j} - \delta_{j-1}) + (\delta_{j-1} - \delta_{i}).
    \]
    For $j=i+1$ and $i > 1$, we can use the decomposition
    \[
      \delta_{i+1} - \delta_{i} 
      =
        \left( - \frac{\alpha_{i+1}-\alpha_i}{\alpha_i-\alpha_{i-1}} \delta_{i-1} +
          \frac{\alpha_{i+1}-\alpha_i}{\alpha_i-\alpha_{i-1}} \delta_{i} \right)
      + \left( \frac{\alpha_{i+1}-\alpha_i}{\alpha_i-\alpha_{i-1}} \delta_{i-1}
        -  \frac{\alpha_{i+1}-\alpha_{i-1}}{\alpha_i-\alpha_{i-1}} \delta_{i} + \delta_{i+1} 
        \right)
    \]
    into $X$-circuits with three non-vanishing components.
    As final consideration for type~(1), 
    the $X$-circuit $\delta_{2}-\delta_{1}$ cannot be written as a conic combination 
    of $X$-circuits with three non-zero entries, because any conic combination of 
    those  $X$-circuits has a positive entry in its non-vanishing component 
    with maximal index.
    For $X$-circuits of type (2) from Example \ref{ex:circuits-univariate1}, simply note that these are also $\R$-circuits.
    Therefore a necessary condition for a type (2) $X$-circuit $\lambda$ to be extremal in $\cone\Lambda_{[0,\infty)}$ is that $\lambda$ belongs to $\Lambda_{\R}^\star(\cA)$.

    It remains to show that none of the remaining $X$-circuits can be written as a convex combination of the others. 
    First note that an $X$-circuit $\nu \in \Lambda_{\R}^\star(\cA)$ 
    cannot be decomposed into a sum which involves an $X$-circuit $\tilde{\nu}$
    with two non-vanishing components. Namely, since $\mathcal{A} \nu = 0$
    and $\mathcal{A} \tilde{\nu} > 0$, we would obtain for the other summand
    $\nu - \tilde{\nu}$ the property $\mathcal{A} (\nu - \tilde{\nu}) < 0$ and
    thus $\sigma_{[0,\infty)}(-\mathcal{A} (\nu - \tilde{\nu})) = \infty$, a contradiction.
    And of course it is trivially true that no element $\lambda \in \Lambda_{\R}^\star(\cA)$ can be written as a convex combination of other such elements.
    Since $\cone\Lambda_{[0,\infty)}$ is finitely generated and there is no $S \subsetneq \{\delta_2-\delta_1\} \cup \Lambda_{\R}^\star(\cA)$ for which $\cone\Lambda_{[0,\infty)} = \cone S$, we conclude that $\{\delta_2-\delta_1\} \cup \Lambda_{\R}^\star(\cA)$ are the reduced $X$-circuits of $\cA$.
\end{example}

The remainder of this section is organized as follows.
Section \ref{subsec:log_dual_valid} proves Theorem \ref{thm:log_dual_x_sage_valid}, which is instrumental in later subsections.
In Section \ref{subsec:topology_of_circuit_graph} we introduce and prove a certain representation result for the CG cone.
Given the groundwork laid in these two subsections, 
Section \ref{subsec:x_circuits_sage_thm1} proves Theorem \ref{thm:reduced_x_circuits} in very short order.
Section \ref{subsec:x_circuits_sage_thm2} proves Theorem \ref{thm:polyhedra_reduced_x_circuits} by refining the arguments from Section \ref{subsec:x_circuits_sage_thm1}.

\subsection{Proof of Theorem \ref{thm:log_dual_x_sage_valid}}\label{subsec:log_dual_valid}

We begin with the following simple lemma.

\begin{lemma}\label{lem:relint_containment}
    If $S \subset T$ are convex sets where $S$ is closed and $S \cap \relint T \neq \emptyset$, then $S = \cl(S \cap \relint T)$.
\end{lemma}
\begin{proof}
    Rockafellar's \cite[Theorem 18.2]{rockafellar-book} states that every relatively open set contained in $T$ is contained in the relative interior of some face of $T$.
    By our assumption $S \cap \relint T \neq \emptyset$, the only face of $T$ which contains $S$ is $T$ itself.
    Since $\relint S$ is obviously relatively open, we have $\relint S \subset \relint T$, and the claim follows by the identity $S = \cl \relint S$ for closed convex sets.
\end{proof}

\begin{proof}[Proof of Theorem \ref{thm:log_dual_x_sage_valid}]
    Use Rockafellar's \cite[Corollary 16.5.2]{rockafellar-book} to invoke Theorem  \ref{thm:primal_x_age_powercone} from a dual point of view, which gives
    $C_X(\cA,\beta)^* = \bigcap C_X(\cA,\lambda)^*$, where the intersection
    runs over all $\lambda \in \Lambda_X(\cA,\beta)$.
    Then Proposition \ref{prop:weighted_dual_x_age_powercone} implies
    \begin{equation}\label{eq:dual_x_sage_powercone}
        C_X(\cA)^* = \left\{ v \in \R^{\cA}_+ \,:\, \forall \, \lambda \in \Lambda_X(\cA),~ \beta \coloneqq \lambda^-,~ \exp(\sigma_X(-\cA \lambda)) \prod_{\alpha \in \lambda^+} v_\alpha^{\lambda_\alpha} \geq v_\beta  \right\}.
    \end{equation}
    We claim that $C_X(\cA)^*$ can be represented as the closure of its intersection with the positive orthant, that is,
    $C_X(\cA)^* = \cl\left( C_X(\cA)^* \cap \R^{\cA}_{++} \right)$.
    Since $C_X(\cA)$ contains all posynomials and is contained in the nonnegativity cone, the dual $C_X(\cA)^*$ contains the moment cone but is still contained in the nonnegative orthant.
    As we have assumed $X$ is nonempty, $C_X(\cA)^*$ must contain a point $\exp(\cA^T x) \in \R^{\cA}_{++}$, so $C_X(\cA)^* \cap \relint \R^\cA_{+} \neq \emptyset$.
    Applying Lemma \ref{lem:relint_containment} with $S = C_X(\cA)^*$ and $T = \R^\cA_+$ gives
    $C_X(\cA)^* = 
    \cl\left( C_X(\cA)^* \cap \relint \R^{\cA}_+ \right) =
    \cl\left( C_X(\cA)^* \cap \R^{\cA}_{++} \right)$.
      
    When considering $C_X(\cA)^*$ only over the positive orthant, the inequalities
    \[
    \exp(\sigma_X(-\cA \lambda)) \prod_{\alpha \in \lambda^+} v_\alpha^{\lambda_\alpha} \geq v_\beta 
    \]
    appearing in \eqref{eq:dual_x_sage_powercone}
    may be rewritten as
    \[
    \textstyle\sum_{\alpha \in \lambda^+} \lambda_\alpha \log v_\alpha - \log v_\beta + \sigma_X(-\cA \lambda) \equiv \phi_\lambda(y)  \geq 0,
    \]
    where we used $\lambda_\beta = -1$ and $y = \log v \in \R^{\cA}$.
   Hence,
    \begin{align*}
        C_X(\cA)^* 
        & = \cl\{ \exp(y)  \,:\,  \phi_\lambda(y) \geq 0 \; \: \forall \, \lambda  \in \Lambda_X(\cA) \} \\
        & = \cl\{ \exp(y)  \,:\, (y,1) ^T (\lambda, \tau) \geq 0 \; \: \forall  \, \lambda \in \Lambda_X(\cA),\, \tau \geq \sigma_X(-\cA \lambda) \} \\
        & = \cl\{ \exp(y)  \,:\, (y,1)^T(\nu,\tau) \geq 0 \; \: \forall \, (\nu,\tau)  \in G_X(\cA) \}.
    \end{align*}
    By the definition of the dual cone from convex analysis, the property
    $(y,1)^T(\nu,\tau) \geq 0 \; \: \forall \, (\nu,\tau)  \in G_X(\cA)$ is the same as 
    $(y,1) \in G_X(\cA)^*$. This completes the proof.
\end{proof}

The ability to represent $C_X(\cA)^*$ in terms of $G_X(\cA)^*$ is key to our proofs of Theorems \ref{thm:reduced_x_circuits} and  \ref{thm:polyhedra_reduced_x_circuits}.
Note that the theorem remains true when $G_X(\cA)$ is replaced by the smaller set $\cone\{\phi_\lambda \,:\, \lambda \in \Lambda_X(\cA) \}$, because the term $(\zerob,1)$ simply requires $(y,t) \in G_X(\cA)^*$ to have $t \geq 0$.

\subsection{Topological properties of the CG cone}\label{subsec:topology_of_circuit_graph}

We need some topological properties of the CG cone from 
Definition~\ref{def:circuit_graph}.

\begin{theorem}\label{thm:circuit_graph_generation}
    $G_X(\cA) = \cone\left(\{ \phi_\lambda \,:\, \lambda \in \Lambda_X^\star(\cA) \}\cup\{(\zerob,1)\}\right)$.
\end{theorem}
The proof of this theorem essentially reduces to showing that $G_X(\cA)$ is pointed and closed.
The pointedness of the CG cone is easy to show, but closedness is a more delicate matter.
In fact -- our proof that $G_X(\cA)$ is closed relies on the fact that it is pointed.
We therefore prove pointedness before discussing closedness any further.

\begin{lemma}\label{lem:circuit_graph_pointed}
   The closure of the CG cone contains no lines.
\end{lemma}
\begin{proof}
    We focus on proving $G_X(\cA)^*$ is full-dimensional.
    Let $|\cA| = m$.
    We assumed at the outset of the article that the moment cone $M_X(\cA) \coloneqq \cone\{\exp(\cA^T x) \,:\, x \in X\}$ was full-dimensional, i.e., $\dim M_X(\cA) = m$; we use that assumption in this lemma.
    Specifically, since $C_X(\cA)$ is contained within the nonnegativity cone, we have that $M_X(\cA) \subset C_X(\cA)^*$ and so $\dim C_X(\cA)^* = m$.
    By Theorem \ref{thm:log_dual_x_sage_valid} and continuity of the exponential function, we see that if $\dim C_X(\cA)^* = m$, then the preimage
    $S \coloneqq \{ y \,:\, (y,1) \in G_X(\cA)^* \}$
    likewise has dimension $m$.
    Consider the induced cone associated with $S$:
    \[
    \indco S = \cl\{(y,t) \,:\, t > 0,\, y/t \in S \} = \cl\{ (y,t) \,:\, t > 0,\, (y,t) \in G_X(\cA)^* \}.
    \]
    The rightmost expression in the above display tells us $\indco S \subset G_X(\cA)^*$.
    We claim without proof that since $S$ is a full-dimensional convex set, $\indco S$ is similarly full-dimensional.
    Taking this claim as given, $\indco S \subset G_X(\cA)^*$ implies $G_X(\cA)^*$ is full-dimensional.
    Because $G_X(\cA)^*$ is full-dimensional, $\cl G_X(\cA) = G_X(\cA)^{**} \supset G_X(\cA)$ contains no lines.
\end{proof}

In the special case where $X$ is a polyhedron, closedness of $G_X(\cA)$ follows from Theorem \ref{thm:polyhedron_x_finite_circuits}, which tells us that $\Lambda_X(\cA)$ is finite.
To prove closedness for arbitrary convex sets $X$ we need to more carefully appeal to properties of the generating set $\{ \phi_\lambda \,:\, \lambda \in \Lambda_X(\cA) \} \cup \{(\zerob,1)\}$.

\begin{lemma}\label{lem:circuit_graph_closed}
    The CG cone is closed.
\end{lemma}
\begin{proof}
    Let $S_\beta = \{ (\lambda,\sigma_X(-\cA \lambda)) \,:\, \lambda \in \Lambda_X(\cA,\beta) \}$.
    By Theorem \ref{thm:characterize_circuits}, the elements $\phi_\lambda \in S_\beta$ are edge generators for the closed convex cone $T_\beta = \cone\{ (\nu,\sigma_X(-\cA \nu)) \,:\, \nu \in N_\beta,\, \sigma_X(-\cA \nu) < \infty\}$.
    From $S_\beta$ we form $S'_\beta \coloneqq \conv S_\beta$, and find
    $S'_\beta$ is isomorphic to 
    $S'_\beta = \{ \phi_\lambda \in T_\beta \,:\, \lambda_{\beta} = -1\}$.
    Because $S_\beta$ is bounded, $S'_\beta$ is likewise bounded.
    Because $S'_\beta$ is a slice of a closed convex cone $T_\beta$, we have that $S'_\beta$ is closed.
    Therefore we conclude $S'_\beta$ is compact.
    
    Now define $S' = (\bigcup_{\beta \in \cA} S'_\beta) \cup \{(\zerob,1)\}$.
    The set $S'$ is a compact generating set for $G_X(\cA)$ which does not contain the origin.
    Since $\cl G_X(\cA)$ is known to contain no lines (Lemma \ref{lem:circuit_graph_pointed}), we apply Proposition \ref{prop:compactly_generated_pointed_cones_closed} to $S'$, $\cone S'$ to infer that $\cone S' = G_X(\cA)$ is closed.
\end{proof}

\begin{proof}[Proof of Theorem \ref{thm:circuit_graph_generation}]
    Lemmas \ref{lem:circuit_graph_pointed} and \ref{lem:circuit_graph_closed} show $G_X(\cA)$ is closed and pointed.
    By \cite[Corollary 18.5.2]{rockafellar-book}, we have that $G_X(\cA)$ may be expressed as the conic hull of any set of vectors containing all of its extreme rays.
    Since $S = \{ \phi_\lambda \,:\, \lambda \in \Lambda_X(\cA)\} \cup \{(\zerob,1)\}$ is a generating set for $G_X(\cA)$, it must contain all extreme rays of $G_X(\cA)$.
    However, by definition of $\Lambda_X^\star(\cA)$, if $\lambda$ does not belong to $\Lambda_X^\star(\cA)$, then $\phi_\lambda \in S$ does not generate an extreme ray of $G_X(\cA)$.
    We may therefore form $T = S \setminus \{ \phi_\lambda \,:\, \lambda \not\in \Lambda_X^\star(\cA) \}$ and still find
    $G_X(\cA) = \cone T$.
    This proves the theorem.
\end{proof}

\subsection{Proof of Theorem \ref{thm:reduced_x_circuits}.}\label{subsec:x_circuits_sage_thm1}

\begin{proof}[Proof of Theorem \ref{thm:reduced_x_circuits}]
    Using the representation $G_X(\cA) = \cone\left(\{ \phi_\lambda \,:\, \lambda \in \Lambda_X^\star(\cA) \}\cup\{(\zerob,1)\}\right)$ provided by Theorem \ref{thm:circuit_graph_generation}, we can express
    \begin{equation}\label{eq:repr_log_dual_reduced_circuits}
        (y,1) \in G_X(\cA)^* \Leftrightarrow (y,1)^T(\lambda,\sigma_X(-\cA \lambda)) \geq 0 ~\forall~ \lambda \in \Lambda_X^\star(\cA).
    \end{equation}
    We obtain the following refinement of Equation \eqref{eq:dual_x_sage_powercone}, by combining \eqref{eq:repr_log_dual_reduced_circuits} with Theorem \ref{thm:log_dual_x_sage_valid}:
    \begin{equation}\label{eq:repr_dual_sage_reduced_circuits}
        C_X(\cA)^* = \left\{ v \in \R^{\cA}_+ \,:\, \forall \, \lambda \in \Lambda_X^\star(\cA),~ \beta \coloneqq \lambda^-,~ \exp(\sigma_X(-\cA \lambda)) \prod_{\alpha \in \lambda^+} v_\alpha^{\lambda_\alpha} \geq v_\beta  \right\}.
    \end{equation}
    Of course, Equation \eqref{eq:repr_dual_sage_reduced_circuits} can be written as $C_X(\cA)^* = \bigcap_{\lambda \in \Lambda_X^\star(\cA)} C_X(\cA,\lambda)^*$. We appeal to conic duality principles (again, \cite[Corollary 16.5.2]{rockafellar-book}) to obtain the claim of the theorem.
\end{proof} 

\subsection{Proof of Theorem \ref{thm:polyhedra_reduced_x_circuits}}\label{subsec:x_circuits_sage_thm2}

A conceptual message from the last section is that it can be very useful to analyze $C_X(\cA)$ in terms of the vectors $y$ where $\exp y$ belongs to $C_X(\cA)^*$.
This section will hammer that message home.
We begin with the lemma that ultimately led us to define $G_X(\cA)$ as per Definition \ref{def:circuit_graph}, rather than as the simpler set $\cone\{\phi_\lambda \,:\, \lambda \in \Lambda_X(\cA)\}$.

\begin{lemma}\label{lem:separate_log_moment_by_circuit}
    If $X$ is polyhedral and $\Lambda \subsetneq \Lambda_X^\star(\cA)$, then there must exist a $\tilde{y} \in \R^{\cA}$ satisfying $\phi_{\lambda'}(\tilde{y}) \geq 0$ for all $ \lambda' \in \Lambda$, yet for some $\lambda \in \Lambda_X^\star(\cA) \setminus \Lambda$ we have $\phi_\lambda(\tilde{y}) < 0$.
\end{lemma}
\begin{proof}
    Let $T_1 = \{ \phi_\lambda \,:\, \lambda \in \Lambda_X^\star(\cA)\} \cup \{ (\zerob,1) \}$ and $T_2 = \{ \phi_\lambda \,:\, \lambda \in \Lambda \} \cup \{(\zerob,1)\}$.
    Of course, a vector $\tilde{y}$ satisfies $\phi_{\lambda'}(\tilde{y}) \geq 0$ for all $\lambda' \in \Lambda$ if and only if $(\tilde{y},1) \in (\cone T_2)^*$.
    We will show that given the polyhedrality of $X$ and the assumption on $\Lambda$, there exists a vector $\tilde{y}$ for which $(\tilde{y},1) \in (\cone T_2)^* \setminus (\cone T_1)^*$.
    The result will follow since membership of vectors $(y,1) \in (\cone T_1)^*$ is equivalent to $\phi_\lambda(y) \geq 0$ for all $\lambda \in \Lambda_X^\star(\cA)$.
    
    Since $X$ is polyhedral, the cones $T_1$ and $T_2$ are also polyhedral (both are finitely generated by Theorem \ref{thm:polyhedron_x_finite_circuits}).
    Meanwhile, Theorem \ref{thm:circuit_graph_generation} tells us that $G_X(\cA) = \cone T_1$, and the definition of reduced circuits is such that every $\phi_\lambda \in T_1 \setminus \{(\zerob,1)\}$ generates an extreme ray in $G_X(\cA)$.
    Since $\Lambda \subsetneq \Lambda_X^\star(\cA)$, there exists a 
    $\phi_\lambda \in T_1 \setminus T_2$ which generates an extreme ray of $G_X(\cA)$.
    Therefore $\cone T_2$ is a strict subset of $\cone T_1 \equiv G_X(\cA)$.
    We may take dual cones to find $(\cone T_2)^* \supsetneq (\cone T_1)^*$.
    Note that since $T_1$ and $T_2$ contain $\{(\zerob,1)\}$, the dual cones must be contained in $K = \R^{\cA} \times \R_+$.
    Furthermore, since $X$ is presumed nonempty, Theorem \ref{thm:log_dual_x_sage_valid} tells us there exists a point $(y,1) \in (\cone T_1)^*$, so the relative interiors of $(\cone T_1)^*$ and $(\cone T_2)^*$ are contained within the relative interior of $K$.
    As our last step, use the fact that if one closed polyhedral cone strictly contains another closed polyhedral cone, then there exists a point in the relative interior of the larger cone which may be separated from the smaller cone; apply this to $(\cone T_2)^* \supsetneq (\cone T_1)^*$ to find a point $(y',t') \in \relint((\cone T_2)^*) \setminus (\cone T_1)^*$ with $t' > 0$.
    From this $(y',t')$ we rescale $\tilde{y} = y' / t'$ so that $(\tilde{y},1) \in (\cone T_2)^* \setminus (\cone T_1)^*$.
\end{proof}

\begin{remark}
    We take a moment to unpack the technical dependencies in 
    Lemma \ref{lem:separate_log_moment_by_circuit}.
    We explicitly cited Theorem \ref{thm:circuit_graph_generation}.
    Our proof of that result relied on Lemma \ref{lem:circuit_graph_closed}, which states that the CG cone is closed, and which we proved by appeal to Theorem \ref{thm:characterize_circuits}.
    However, when $X$ is a polyhedron, Lemma \ref{lem:circuit_graph_closed} can alternatively be proven by appeal to Theorem \ref{thm:polyhedron_x_finite_circuits}.
\end{remark}

Our next lemma shows how to take a condition stated in terms of Lemma \ref{lem:separate_log_moment_by_circuit}, and deduce a statement about $C_X(\cA)^*$.
The lemma's proof requires only that $X$ be nonempty and convex.

\begin{lemma}\label{lem:exponentiate_separating_hyperplane}
    If $\tilde{y} \in \R^{\cA}$ satisfies $\phi_\lambda(\tilde{y}) < 0$ for some $\lambda \in \Lambda_X(\cA)$, then $\exp \tilde{y} \not\in C_X(\cA)^*$.
\end{lemma}
\begin{proof}
    We will find a vector $z \in \R^{\cA}$ where $0 \leq z^T \exp y$ for all $\exp y \in C_X(\cA)^*$, and yet $z^T \exp \tilde{y} < 0$.
    By continuity, the condition that $0 \leq z^T \exp y$ for all $\exp y \in C_X(\cA)^*$ will imply the slightly stronger statement that $0 \leq z^T v$ for all $v \in C_X(\cA)^*$.
    Therefore $z$ will evidently serve as a separating hyperplane to prove the desired claim.
    Let $\beta \coloneqq \lambda^-$.
    
    Since $\lambda \in \Lambda_X(\cA)$, Theorem \ref{thm:log_dual_x_sage_valid} says that $\phi_\lambda(y) \geq 0$ whenever $\exp y \in C_X(\cA)^*$.
    Combine $\phi_\lambda(\tilde{y}) < 0$ with strict monotonicity of the exponential function to conclude
    \begin{equation}\label{eq:sep_hyp_y}
    \exp(\phi_\lambda(\tilde{y})) < 1 \leq \exp(\phi_\lambda(y)) \quad\text{ for all }\quad \exp y \in C_X(\cA)^*.
    \end{equation}
    Notice that taking a difference $\phi_\lambda(y) - \phi_\lambda(\tilde{y}) = \lambda_{\setminus \beta}^T(y_{\setminus \beta} + \tilde{y}_{\setminus \beta}) - y_\beta + \tilde{y}_\beta$ eliminates the support function term appearing in $\phi_\lambda$.
    Defining $u = \phi_\lambda(\tilde{y})$, we multiply both sides of the non-strict inequality in \eqref{eq:sep_hyp_y} by $\exp(-u-\tilde{y}_\beta+y_\beta)$ to obtain
    \begin{align}
    0 &\leq \exp\left(\lambda_{\setminus \beta}^T(y_{\setminus \beta}-\tilde{y}_{\setminus \beta})\right) - \exp(-u-\tilde{y}_\beta+y_\beta). \label{eq:sep_hyp_expy_1}
    \end{align}
    Convexity of the exponential function tells us that $\exp\left(\lambda_{\setminus \beta}^T(y_{\setminus \beta}-\tilde{y}_{\setminus \beta})\right) \leq \lambda_{\setminus \beta}^T\exp(y_{\setminus \beta}-\tilde{y}_{\setminus \beta})$,
    where the right-hand-side may be rewritten using the Hadamard product
    \[
    \lambda_{\setminus \beta}^T\exp(y_{\setminus \beta}-\tilde{y}_{\setminus \beta}) = 
    \left(\lambda_{\setminus \beta} \circ \exp(-\tilde{y}_{\setminus \beta})\right)^T\exp(y_{\setminus \beta}).
    \]
    Applying these observations to \eqref{eq:sep_hyp_expy_1} gives
    \begin{equation}\label{eq:sep_hyp_expy_2}
        0 \leq \left(\lambda_{\setminus \beta} \circ \exp(-\tilde{y}_{\setminus \beta})\right)^T\exp(y_{\setminus \beta}) - (\exp(-u-\tilde{y}_\beta))\exp(y_\beta).
    \end{equation}
    
    Inequality \eqref{eq:sep_hyp_expy_2} is essentially what we need to prove the lemma.
    Defining $z \in \R^\cA$ by $z_\alpha = \lambda_\alpha \exp(-\tilde{y}_\alpha)$ for $\alpha \neq \beta$ and $z_\beta = -\exp(-u-\tilde{y}_\beta)$, we have that $0 \leq z^T \exp y$ for all $\exp y \in C_X(\cA)^*$.
    As explained at the beginning of this proof, we appeal to continuity to establish $0 \leq z^T v$ for all $v \in C_X(\cA)^*$.
    One may use $\lambda_{\setminus \beta}^T 1 = 1$ to trivially evaluate $z^T \exp(\tilde{y}) = 1 - \exp(-u)$, and since $u < 0$ by assumption on $\tilde{y}$, we conclude $z^T\exp(\tilde{y}) < 0$.
\end{proof}

\begin{proof}[Proof of Theorem \ref{thm:polyhedra_reduced_x_circuits}]
    By Theorem \ref{thm:log_dual_x_sage_valid}, we have the dual description $C_X(\cA)^* = \cl \{ \exp y \, : \, (y,1) \in G_X(\cA)^*\}$.
    Applying Theorem \ref{thm:circuit_graph_generation} then gives
    \begin{align*}
      C_X(\cA)^* & = \cl \{ \exp y \, : \, \phi_\lambda(y) \geq 0 \; \: \forall \, \lambda
      \in \Lambda_X^\star(\cA) \}.
    \end{align*}
    We rewrite the condition on $\phi_\lambda(y)$ as a condition on $v = \exp y$ using the power-cone formulation in Proposition \ref{prop:weighted_dual_x_age_powercone}.
    Since $X$ is polyhedral, Theorem \ref{thm:polyhedron_x_finite_circuits} tells us there are finitely many normalized $X$-circuits $\Lambda_X(\cA)$.
    We may therefore express $C_X(\cA)^*$ as a finite intersection of dual $\lambda$-witnessed AGE cones,
    \[
      C_X(\cA)^* = \bigcap_{\lambda \in \Lambda_X^\star(\cA)}
      C_X(\cA,\lambda)^*.
    \]
    Moreover, each dual $\lambda$-witnessed AGE cone $C_X(\cA,\lambda)^*$ is an outer-approximation of the full-dimensional moment cone $\cone\{ \exp(\cA^T x) \,:\, x \in X \}$, hence there exists a point $v_0$ in the interior of the moment cone where $v_0 \in \interior C_X(\cA,\lambda)^*$ for all $\lambda \in \Lambda_X^\star(\cA)$.
    Therefore, by \cite[Corollary 16.4.2]{rockafellar-book} we have
    \[
        C_X(\cA) = (C_X(\cA)^*)^* = \sum_{\lambda \in \Lambda_X^\star(\cA)} (C_X(\cA,\lambda)^*)^* = \sum_{\lambda \in \Lambda_X^\star(\cA)} C_X(\cA,\lambda),
    \]
    which establishes the first part of the theorem.
    
    For the second part of the theorem, suppose $\Lambda$ is a proper subset of $\Lambda_X^\star(\cA)$.
    Consider the set $C = \sum_{\lambda \in \Lambda} C_X(\cA,\lambda)$ and its
    dual $C^* = \bigcap\{ C_X(\cA,\lambda)^* \,:\, \lambda \in \Lambda \}$.
    Clearly, since $C \subset C_X(\cA)$ we have $C^* \supset C_X(\cA)^*$ -- we will show that this containment is \textit{strict}, i.e., $C^* \supsetneq C_X(\cA)^*$.
    Once this is done, duality will tell us that $C \subsetneq C_X(\cA)$.
    
    Since $C$ is contained within the signomial nonnegativity cone we again have that $C^*$ contains the moment cone and so by Lemma \ref{lem:relint_containment} we have $C^* = \cl(C^* \cap \R^{\cA}_{++})$.
    Work with $C^*$ over the positive orthant using Proposition \ref{prop:weighted_dual_x_age_powercone} to express it as $C^* = \cl\{ \exp y \,:\, y \in Y\}$ for $Y \coloneqq \{ y \,:\, \phi_\lambda(y) \geq 0 \; \: \forall \, \lambda \in \Lambda \}$.
    By Lemma \ref{lem:separate_log_moment_by_circuit} there exists an element $\tilde{y} \in Y$ for which some $\lambda \in \Lambda_X^\star(\cA) \setminus \Lambda$ satisfies $\phi_\lambda(\tilde{y}) < 0$.
    Apply Lemma \ref{lem:exponentiate_separating_hyperplane} to this pair $(\phi_\lambda, \tilde{y})$ to see that $\exp \tilde{y}$ can be separated from the closed convex set $C_X(\cA)^*$.
    We have therefore found a point $\tilde{y}$ where $\exp \tilde{y} \in C^*$ and yet $\exp \tilde{y}$ can be separated from $C_X(\cA)^*$, so we conclude $C^* \supsetneq C_X(\cA)^*$.
\end{proof}

Before concluding this section we would like to point out a more general way to frame our analysis.
Given a pair $(\lambda,a) \in \R^m \times \R$ where $\lambda$ sums to zero and has exactly one negative component $\lambda_i = -1$, we have a power cone constraint $v_i \leq \exp(a) \prod_{j\neq i}v_j^{\lambda_j}$ which may be rewritten to $1 \leq \exp(a) v^{\lambda}$.
Given a \textit{set} of such pairs $P \subset \R^m \times \R$, we obtain the convex set
\[
    F(P) = \left\{ v \in \R^m_{+} \,:\, 1 \leq \exp(a) v^{\lambda} ~\forall~ (\lambda, a) \in P \right\}.
\]
We have effectively shown that if $K = \cone(P \cup \{(\zerob,1)\})$ is pointed and $F(P)$ intersects the positive orthant, then the unique minimum $P^\star \subset P$ for which $F(P^\star) = F(P)$ can be read off from the extreme rays of the polyhedral cone $K$.

\section{Extreme rays of half-line SAGE cones}\label{sec:primal_perspective}

In the previous section, we showed that by appropriate appeals to convex duality, one may derive representations of $C_X(\cA)$ with little to no redundancy. Here we build upon those results to completely characterize the extreme rays of the $X$-SAGE cone for the univariate case $X=[0,\infty)$.

\begin{proposition}\label{prop:extremalsupportsconicdim1_alt}
	For $\alpha_1 < \cdots < \alpha_m$, the extreme rays of $C_{[0,\infty)}(\{\alpha_1,\ldots,\alpha_m\})$ are:
	\begin{itemize}
		\item[(1)] $\R_+\cdot \exp(\alpha_1 x) $,
		\item[(2)] $\R_+\cdot \{\exp(\alpha_2 x)-\exp(\alpha_1 x)\}$,
		\item[(3)] $\R_+ \cdot \{c_{i+1} \exp(\alpha_{i+1}x)+
		c_{i} \exp(\alpha_i x)+c_{i-1}\exp(\alpha_{i-1}x) \, : \, 
		2 \le i \le m-1 \}$ with
		\[
		c_{i+1} > 0, \qquad c_{i-1} > 0, \quad \text{and} \quad c_{i} = - \left( \frac{c_{i-1}}{\lambda_{i-1}} \right)^{\lambda_{i-1}}
		\left( \frac{c_{i+1}}{\lambda_{i+1}} \right)^{\lambda_{i+1}},
		\]
		where
		\[
		\lambda_{i+1} = \frac{\alpha_i- 
			\alpha_{i-1}}{\alpha_{i+1}-\alpha_{i-1}}, \quad 
		\lambda_{i-1} = \frac{\alpha_{i+1}- \alpha_i}{\alpha_{i+1}-\alpha_{i-1}}, \quad\text{and}\quad \frac{c_{i-1}}{c_{i+1}} \geq \frac{\lambda_{i-1}}{\lambda_{i+1}}.
		\]
	\end{itemize}
\end{proposition}
\begin{proof}
    Let $\cA = \{\alpha_1,\ldots,\alpha_m\}$.
	By Theorem \ref{thm:polyhedra_reduced_x_circuits}, all edge generators of $C_{[0,\infty)}(\cA)$ are either monomials or $\lambda$-witnessed AGE functions where $\lambda$ is a reduced $[0,\infty)$-circuit.
	By Example \ref{ex:reduced_extreme_rays}, $\Lambda_{[0,\infty)}^\star(\cA) = \{\delta_2 - \delta_1\} \cup \Lambda_{\R}^\star(\cA)$.
	Since $n=1$, Proposition \ref{prop:x_circuit_aff_indep} says all circuits $\lambda$ have $|\supp \lambda| \leq 3$.
	We therefore divide the proof into considering cases of monomials, and $X$-AGE functions with two or three terms.
	
	First we address the monomials.
	Given $f(x) = \exp(\alpha_i x)$ with $i > 1$, we can write $f = f_1 + f_2$ with $f_1(x) = \exp(\alpha_{i}x) - \exp(\alpha_{i-1}x)$ and $f_2(x) = \exp(\alpha_{i-1}x)$ -- 
	the summand $f_1$ is nonnegative on $[0,\infty)$ because $\alpha_i > \alpha_{i-1}$, and $f_2$ is globally nonnegative.
	Therefore the only possible extremal monomial in $C_{[0,\infty)}(\cA)$ is $f(x) = \exp(\alpha_1 x)$.
	Since $X = [0,\infty)$, the leading term of any $g \in C_X(\cA)$ must have positive coefficient.
	Moreover, if $g$ is not proportional to $f$, the leading term of $g$ must have exponent greater than $\alpha_1$.
	Therefore any convex combination of AGE functions $g \in C_{[0,\infty)}(\cA)$ which are not proportional to $f$ must disagree with $f(x)$ in the limit as $x$ tends to infinity. 
	We conclude $f$ is extremal in $C_{[0,\infty)}(\cA)$.
	
	Now we consider the 2-term case, where, by
	Example \ref{ex:reduced_extreme_rays}, we have to consider signomials of the
	form
	$f(x) = c_2 \exp(\alpha_2 x) - c_1 \exp(\alpha_1 x)$. We observe that $f$
	is nonnegative on $[0,\infty)$ if and only if $c_2 \geq c_1 \geq 0$, and furthermore that such signomials are nonextremal unless $c_1 = c_2$. 
    To see that $f(x) = \exp(\alpha_2 x) - \exp(\alpha_1 x)$ is indeed extremal, note that $f$ cannot be written as a convex combination involving any 3-term AGE functions, because any conic combination of 3-term AGE functions has a leading term with positive coefficient on $\exp(\alpha_i x)$ for some $i \geq 3$.
	
	We have already proven cases (1) and (2) of the proposition. 
	Using Example \ref{ex:reduced_extreme_rays}, we know that any extremal 3-term $X$-AGE function belongs to a $\lambda$-witnessed AGE cone where $\lambda$ is a reduced $\R$-circuit.
	These reduced $\R$-circuits have the property $\supp \lambda = \{i-1,i,i+1\}$ $\alpha_{i-1}\lambda_{i-1}+\alpha_{i+1}\lambda_{i+1} = \alpha_i$, $\lambda_i = -1$.
	Any $X$-AGE function with such a witness is nonnegative on all of $\R$.
	Therefore any 3-term $X$-AGE function $f$ that is \textit{extremal} in $C_{[0,\infty)}(\cA)$ is also extremal in $C_{\R}(\cA) \subset C_{[0,\infty)}(\cA)$, which (by  \cite[Prop. 4.4]{Katthaen:Naumann:Theobald:UnifiedFramework}) implies
	\begin{equation}\label{eq:univariate_extremal_age_sig}
	    f(x) = c_{i+1} \exp({\alpha_{i+1} x}) - \left(\left[\frac{c_{i+1}}{\lambda_{i+1}}\right]^{\lambda_{i+1}}\left[\frac{c_{i-1}}{\lambda_{i-1}}\right]^{\lambda_{i-1}}\right)\exp(\alpha_i x) +  c_{i-1} \exp(\alpha_{i-1} x).
	\end{equation}
	
	We have arrived at the final phase of proving part (3) of this proposition.
	By the equality case in the AM/GM inequality and using $\exp(\alpha_i x) = \left(\exp(\alpha_{i+1} x)^{\lambda_{i+1}}\right)\left(\exp(\alpha_{i-1} x)^{\lambda_{i-1}}\right)$, one finds the unique minimizer $x^\star$ for functions \eqref{eq:univariate_extremal_age_sig} satisfies
	\[
	\left[\frac{c_{i+1}\exp(\alpha_{i+1} x^\star)}{\lambda_{i+1}}\right] = \left[\frac{c_{i-1} \exp(\alpha_{i-1} x^\star)}{\lambda_{i-1}}\right] \quad\Leftrightarrow\quad x^\star = \ln\left(\frac{c_{i-1}}{c_{i+1}}\frac{\lambda_{i+1}}{\lambda_{i-1}}\right)/(\alpha_{i+1} - \alpha_{i-1}).
	\]
	If $V_i(\lambda,c) \coloneqq (c_{i-1}\lambda_{i+1})/(c_{i+1}\lambda_{i-1})$ satisfies $V_i(\lambda,c) < 1$, then $x^\star < 0$ and by continuity we have $\inf\{ f(x) \,:\, x \geq 0 \} > 0$ -- hence the condition $V_i(\lambda,c) \geq 1$ is necessary for extremality.
	Furthermore, if $V_i(\lambda,c) > 1$, then the unique minimizer of $f$ given by \eqref{eq:univariate_extremal_age_sig} occurs at $x^\star > 0$.
	Such $f$ cannot be decomposed as a convex combination which involves 1-term or 2-term AGE functions (which have $f(x) > 0$ for $x > 0$), and cannot be written as a convex combination consisting solely of 3-term AGE functions \cite[Proposition 4.4]{Katthaen:Naumann:Theobald:UnifiedFramework}, therefore any $f$ given by \eqref{eq:univariate_extremal_age_sig} with $V_i(\lambda,c) > 1$ is extremal in $C_{[0,\infty)}(\cA)$.
	All that remains is to show extremality of functions \eqref{eq:univariate_extremal_age_sig} with $V_i(\lambda,c) = 1$. This follows from the same argument as $V_i(\lambda,c) > 1$, but we must use the stationarity condition $f'(0) = 0$ to preclude using 2-term extremal AGE functions in a decomposition of $f$.
\end{proof}

\section{Discussion and Conclusion}
\label{se:discussion}

In this article we have introduced a convex-geometric notion of an $X$-circuit, which mediates a relationship between point sets $\cA \subset \R^n$ and convex sets $X \subset \R^n$.
By showing that this notion of an $X$-circuit allows an alternative construction of $X$-SAGE cones (Theorems \ref{thm:primal_x_age_powercone} and \ref{thm:reduced_x_circuits}) which cannot be relaxed (Theorem \ref{thm:polyhedra_reduced_x_circuits}), we have demonstrated that conditional SAGE cones exhibit a substantially richer theory than ordinary SAGE cones.
An essential property of this theory is that for general sets $X$ it is not possible to recover an $X$-circuit $\lambda \in \Lambda_X(\cA,\beta)$ given only information on the signs of its components.
As a consequence of this last point -- it is not possible to arrive at the concept of conditional SAGE certificates while relying on a ``circuit number'' approach using only the support of a given polynomial or signomial.

Two lines of theoretical investigations stand out for future work.
First, there is the task of formally situating $X$-circuits in the context of matroid theory (in the case when $X$ is a polyhedron).
Here one can use an interpretation from Theorem \ref{thm:polyhedron_x_finite_circuits}, that $X$-circuits $\lambda \in \Lambda_X(\cA,\beta)$ are outer normal vectors to facets of $-\cA^T X +  N_\beta^\circ$.
A broader area of follow-up work is in-depth analysis of multiplicatively-convex sets $S \subset \R^m_+$ for which $\log(S) = \{ t \,:\, \exp t \in S \}$ is convex.
Some properties of this class of sets include closure under intersection, and closure under the induced-cone operation.

It is of interest to explore the use of the cones $C_X(\cA,\lambda)$ when $\lambda$ is not an $X$-circuit.
Given a signomial $\sum_{\alpha\in\cA}c_{\alpha}\e^{\alpha}$ with numerical $X$-SAGE certificate $\{(c^{(\beta)},\nu^{(\beta)})\}_{\beta\in\cA}$, $c^{(\beta)}\in C_X(\cA,\beta)$, $c \approx \sum_{\beta\in\cA} c^{(\beta)}$, one could refine this certificate to higher precision by solving the power-cone program to decompose $c$ as a sum of vectors in $C_X(\cA,\lambda^{(\beta)})$ for $\lambda^{(\beta)} = \nu^{(\beta)}/|\nu^{(\beta)}_{\beta}|$. 
This would be helpful for large scale problems where $\{(c^{(\beta)},\nu^{(\beta)})\}_{\beta\in\cA}$ is computed with a first-order solver, or when $X$ is an especially complicated spectrahedron.
In the latter case, the standard description of $C_X(\cA)$ would be a mixed semidefinite and relative entropy program, while the formulations for $C_X(\cA,\lambda^{(\beta)})$ would be pure power cone programs. 

The two obstacles to using Theorem \ref{thm:polyhedra_reduced_x_circuits} in computation are that $|\Lambda_X^\star(\cA)|$ can be exponential in $|\cA|$ even when $X=\R^n$, and that finding $X$-circuits requires a procedure to identify extreme rays of a polyhedral cone.
It is not known how severe this first problem is in practice.
For the second problem one could focus on $X$-SAGE polynomials where $X=[-1,1]^n$ or $X=[0,1]^n$.
The cones of such polynomials on $\cA \subset \N^n$ are represented by $C_Y(\cA)$ for $Y = \{ y \in \R^{n} : y \leq \zerob \}$, and finding $\Lambda_Y^\star(\cA)$ is made easier by the fact that $Y$ is a cone. 
The main benefit of this approach for polynomials is the prospect of computing conditional SAGE decompositions in exact arithmetic, especially for sparse polynomials of high degree.

We conclude by noting that although the ``pure'' conditional SAGE methodology is used only for convex constraint sets, additional nonconvex constraints can be accommodated with algebraic techniques.
This can partly be seen in the original work of Chandrasekaran and Shah \cite{chandrasekaran-shah-2016} and more so in the recent work \cite{dressler-murray-2021}.

\bibliography{bibConstrainedSage}
\bibliographystyle{plain}

\section{Appendix}

\subsection{Propositions regarding convex analysis}

The following pair of propositions are used in the proofs of
Theorem~\ref{thm:primal_x_age_powercone}
and Lemma~\ref{lem:circuit_graph_closed}.

\begin{proposition}\label{prop:misc:powercone_relentr}
For fixed $\lambda$ in the interior of the $m$-dimensional probability simplex and $c = (c_0,c_1,\ldots,c_m) \in \R^{m+1}$ with $(c_1,\ldots,c_m) \geq \zerob$, we have
\[
-c_0 \leq \prod_{i=1}^m \left[ c_i / \lambda_i  \right]^{\lambda_i} \quad \Leftrightarrow \quad \text{ some } \nu \in \R^m_+ \text{ satisfies } \nu \parallel \lambda \text{ and } \relentr(\nu,c_{\setminus 0}) - \mathds{1}^T \nu \leq c_0
\]
-- where $\nu \parallel \lambda$ means $\nu$ is proportional to $\lambda$.
\end{proposition}
\begin{proof}
The claim is trivial when $c_0 \geq 0$, and so we consider $c_0 < 0$. Note that in this case, $\prod_{i=1}^m \left[ c_i / \lambda_i \right]^{\lambda_i}$ must be positive, and $\relentr(\nu,c_{\setminus 0})$ must be finite: both of these conditions occur precisely when $c_i > 0$ for all $1 \leq i \leq m$. We therefore can rewrite $-c_0 = |c_0 | \leq \prod_{i=1}^m \left[ c_i / \lambda_i \right]^{\lambda_i}$ as $1 \leq \prod_{i=1}^m \left[ c_i / (|c_0|\lambda_i) \right]^{\lambda_i}$, and by taking the log of both sides, obtain $\relentr(\nu,c_{\setminus 0}) - \mathds{1}^T \nu \leq c_0$ for $\nu = |c_0|\lambda$. For the other direction, one may write the proportionality relationship $\nu \parallel \lambda$ as $\nu = s \lambda$, and minimize $\relentr(s \lambda,c_{\setminus 0}) - s$ over $s \geq 0$ to obtain $-\prod_{i=1}^m \left[ c_i / \lambda_i  \right]^{\lambda_i}$.
\end{proof}

\begin{proposition}\label{prop:compactly_generated_pointed_cones_closed}
    Suppose $S \subset \R^{m} \setminus \{\zerob\}$ is compact (not necessarily convex) and set $T = \cone S$.
    If it is known a-priori that $\cl T$ contains no lines, then $T = \cl T$ is closed.
\end{proposition}
\begin{proof}
    Since $\cl T$ is pointed, there exists a distinguished element $t^\star \in T$ for which $(t^\star)^T t > 0$ for all $t \in (\cl T) \setminus \{\zerob\}$.
    Consider the set $H = \{ t \in T \,:\, (t^{\star})^T t = 1 \}$ -- it is clear that $H$ is bounded, $\cone H = T$, and $0 \not\in H$.
    If $H$ is closed, then by \cite[Corollary 9.6.1]{rockafellar-book} we will have that $\cone H = T$ is also closed.
    We show that $H$ is closed by directly considering sequences in $H$.
    We express these sequences with the help of the $m$-fold Cartesian product $S^{m} = S \times \cdots \times S$.
    
    Let $(h^{(k)})_{k \in \N} \subset H$ have a limit in $\R^m$.
    Since $H$ is of dimension at most $m-1$ and is generated by $S$, Carath\'eodory's Theorem tells us that there exists a vector $\lambda^{(k)} \in \R^m_+$ and a block vector $q^{(k)} = (s^{(k)}_1,\ldots,s^{(k)}_m) \in S^m$ where
    \[
     h^{(k)} = \textstyle\sum_{i=1}^m\lambda^{(k)}_i s^{(k)}_i.
    \]
    Since $S$ is compact, the continuous function $s \mapsto (t^\star)^T s$ attains a minimum on $s^\star \in S$ -- since $S$ does not contain zero, we have that $(t^\star)^T(s^\star) = a > 0$.
    It follows that each $\lambda^{(k)}_i$ appearing in the expression for $h^{(k)}$ is bounded above by $1/a < \infty$.
    The sequences $(\lambda^{(k)})_{k \in \N} \subset [0,1/a]^m$ and $(q^{(k)})_{k \in \N} \subset S^m$ are bounded, and therefore $((\lambda^{(k)},q^{(k)}))_{k \in \N}$ has a convergent subsequence.
    The limits $\lambda^{(\infty)}$ and $q^{(\infty)}$ of these convergent subsequences must belong to $[0,1/a]^m$ and $S^m$, respectively.
    By continuity, we have
    \[
    h^{(\infty)} \coloneqq \lim_{k\to\infty} h^{(k)} = \textstyle\sum_{i=1}^m\lambda^{(\infty)}_i s^{(\infty)}_i,
    \]
    hence $h^{(\infty)} \in H$.
    Since we have shown that all convergent sequences in $H$ converge to a \textit{point in $H$}, we have that $H$ is closed.
\end{proof}

Our next proposition is provided for the reader's convenience.
\begin{proposition}\label{prop:convex_analysis:dual_of_linear_image}
    Let $X \subset \R^n$ be a convex cone and consider a matrix $A$ in $\R^{n \times m}$.
    We have $(A^T X)^* = \ker A + A^\dagger X^*$, where $A^\dagger \in \R^{m \times n}$ is the Moore-Penrose pseudo-inverse of $A$.
\end{proposition}
\begin{proof}
    Because $A^T X$ is contained in the subspace $\range A^T$, its dual cone is invariant under translation by vectors in the orthogonal complement $(\range A^T)^\perp = \ker A$.
    In particular, $(A^T X)^* = \ker A + K$ for a convex cone $K \subset \range A^T$.
    We need to show that $K = A^\dagger X^*$.
    
    The definition of the Moore-Penrose pseudo-inverse ensures that $y \in \range A^T$ holds if and only if $A^\dagger A y = y$.
    We can therefore compute $K$ as follows
    \begin{align*}
        K &= \{ y \in \R^m \,:\, y^T z \geq 0 \,\forall\, z \in A^T X \} \cap \{ y \,:\, A^\dagger A y = y\} \\
          &= \{ A^\dagger A y \,:\, (A y)^T x \geq 0 \,\forall\, x \in X \} \cap \{ y \,:\, A^\dagger A y = y\} \\
          &= \{ A^\dagger w \,:\, w \in\R^n, w^T x \geq 0 \,\forall\, x \in X \} \cap (\range A^T) \\
          &= \{ A^\dagger w \,:\, w \in X^* \}.
    \end{align*}
    The transitions from line to line are as follows.
    First, substitute $A^\dagger A y$ for $y$, express $z = A^T x$ for some $x \in X$, and rewrite $y^T (A^T x) = (A y)^T x$.
    Then, substitute $w \coloneqq A y$ and simplify the expression for the range of $A^T$.
    Finally, apply the definition of the dual cone $X^*$ and use the pseudo-inverse identity $A^\dagger A A^\dagger x = A^\dagger x$ for all $x \in \R^n$.
\end{proof}

\subsection{Definitions from convex analysis}

A \emph{face} of a convex set $S \subset \R^n$ is any closed convex $F \subset S$ with the following property: if the line segment $[s_1,s_2] \coloneqq \{ \lambda s_1 + (1-\lambda) s_2 \,:\, 0 \leq \lambda \leq 1\}$ is contained in $S$ and the relative interior of $[s_1,s_2]$ hits $F$, then the entirety of $[s_1,s_2]$ is contained in $F$.
The \emph{dimension} $\dim S$ of a convex set $S$ is the dimension of the smallest affine space containing $S$.
Every nonempty convex set $S$ has a nonempty \textit{relative interior} $\relint S$, which is the interior of $S$ under the topology induced by its affine hull.
A set $K \subset \R^n$ is called a \emph{cone} if it is closed under dilation: $\{ \lambda x \,:\, x \in K\} \subset K$ for all $\lambda > 0$.
The \emph{extreme rays} of a pointed convex cone $K$ are its faces of dimension one.
To any convex cone $K$ we associate the \emph{dual cone} $K^* \coloneqq \{ y\,:\, y^T x \geq 0 \,\forall\,x\in K\}$ and the \textit{polar} $K^\circ = -K^*$.
The \emph{conic hull} of a set $S$, denoted $\cone S$, is the set formed by adjoining the origin to the smallest convex cone containing $S$.

\end{document}